\documentclass[onefignum,onetabnum]{siamart171218}

\usepackage{amsfonts}
\usepackage{graphicx}
\usepackage{epstopdf}
\usepackage{algpseudocode}
\usepackage{subcaption} 
\usepackage{hyperref}
\usepackage{cleveref}

\usepackage{amssymb}
\usepackage{bm}
\ifpdf
  \DeclareGraphicsExtensions{.eps,.pdf,.png,.jpg}
\else
  \DeclareGraphicsExtensions{.eps}
\fi

\newsiamremark{remark}{Remark}
\newsiamremark{hypothesis}{Hypothesis}
\crefname{hypothesis}{Hypothesis}{Hypotheses}
\newsiamthm{claim}{Claim}

\headers{Equilibrated Averaging Residual Method}{Cuiyu He}

\title{A Unified Equilibrated Flux Recovery Framework with Robust A Posteriori Error Estimation
}

\author{Cuiyu He\thanks{Department of Mathematics, University of Georgia, Athens, GA
  (\email{cuiyu.he@uga.edu}).}}

\usepackage{amsopn}

\renewcommand{\theequation}{\thesection.\arabic{equation}}
\def\@eqnnum{{\reset@font\rm (\theequation)}}

\def\abstract{
\advance \rightskip by 10mm
\advance \leftskip by 10mm
\vspace{-0.8em}
\noindent
\small{\bf Abstract.}
}

\def\XXint#1#2#3{{\setbox0=\hbox{$#1{#2#3}{\int}$}
\vcenter{\hbox{$#2#3$}}\kern-.5\wd0}}

\def\a{\alpha}

\renewcommand\o{\omega}\renewcommand\O{\Omega}
\def\S{\Sigma}

\newcommand{\bsigma}{\mbox{\boldmath$\sigma$}}
\newcommand{\btau}{\mbox{\boldmath$\tau$}}

\newcommand{\bftau}{\boldsymbol {\tau}}

\def\bn{{\bf n}}

\newcommand{\bx}{\mbox{\boldmath$x$}}

\def\cE{{\cal E}}

\def\cT{{\cal T}}

\def\sD{{_D}}

\def\f12{\frac12}
\def\dfrac{\displaystyle\frac}

\def\p{\partial}

\newcommand{\gradt}{\nabla\cdot}

\def\osc{{\rm osc \, }}

\def\divvr{{\rm div}}

\def\curll{{\rm curl}}

\newcommand{\tri}{|\!|\!|}

\newcommand{\bdm}{\begin{displaymath}}
\newcommand{\edm}{\end{displaymath}}
\newcommand{\beq}{\begin{equation}}
\newcommand{\eeq}{\end{equation}}
\newcommand{\beqa}{\begin{eqnarray}}
\newcommand{\eeqa}{\end{eqnarray}}
\newcommand{\beqas}{\begin{eqnarray*}}
\newcommand{\eeqas}{\end{eqnarray*}}

\ifpdf
\hypersetup{
  pdftitle={ON-EARM: A General Approach to Conservative Flux Recovery
},
  pdfauthor={C. He}
}
\fi

\def\cN{{\cal N}}

\newtheorem{example}{Example}[section]

\newcommand{\jump}[1]{[\![ #1]\!]}

\begin{document}

\maketitle


\begin{abstract}
We introduce the Equilibrated Averaging Residual Method (EARM), a unified
equilibrated flux-recovery framework for elliptic interface problems that
applies to a broad class of finite element discretizations. The method is
applicable in both two and three dimensions and for arbitrary polynomial
orders, and it enables the construction of computationally efficient
recovered fluxes.
We develop EARM for both discontinuous Galerkin (DG) and conforming finite
element discretizations. For DG methods, EARM can be applied directly and
yields an explicit recovered flux that coincides with state-of-the-art
conservative flux reconstructions.

For conforming discretizations, we further propose the Orthogonal
Null-space--Eliminated EARM (ON-EARM), which ensures uniqueness by
restricting the correction flux to the orthogonal complement of the
divergence-free null space. We prove local conservation and establish a
robust a~posteriori error estimator for the recovered flux in two
dimensions, with robustness measured with respect to jumps in the diffusion
coefficient. Numerical results in two and three dimensions confirm the
theoretical findings.

\end{abstract}




\section{Introduction}\label{intro}


An accurate and locally conservative recovered flux (also referred to as an
equilibrated flux) for finite element (FE) solutions plays a central role in
many applications, including a posteriori error estimation
\cite{Ai:07b,braess2008equilibrated,Ve:09,BFH:14,ErnVo2015,CaCaZh:20,cai2021generalized},
a key component of adaptive mesh refinement, compatible transport in
heterogeneous media \cite{odsaeter2017postprocessing}, and velocity
reconstruction in porous media
\cite{ern2009accurate,vohralik2013posteriori,bastian2014fully,capatina2016nitsche},
among others.

In this work, we focus on equilibrated flux recovery and its associated
a~posteriori error estimation for elliptic interface problems, allowing for
large jumps in the diffusion coefficient. Equilibrated estimators have
attracted significant attention due to their guaranteed reliability for the
conforming error of the finite element solution: specifically, the
reliability constant of an equilibrated a~posteriori error estimator equals
one for the conforming error. This property makes such estimators
particularly well-suited for discretization error control on both coarse and
fine meshes.

The mathematical foundation of equilibrated a~posteriori error estimation
for conforming finite element approximations is the Prager--Synge identity
\cite{prager1947approximations}, which is valid for \(H^{1}(\Omega)\)-conforming
approximations. A generalized Prager--Synge identity for piecewise
\(H^{1}(\Omega)\) (possibly discontinuous) finite element approximations in
both two and three dimensions was established in \cite{cai2021generalized}.
There, the error is decomposed into conforming and nonconforming components,
and the conforming part is guaranteed to be bounded by an equilibrated error
estimator, given by the energy norm of the difference between the numerical
flux and a recovered equilibrated flux.

%
In this paper, we introduce a unified equilibrated flux-recovery framework
for a broad class of finite element methods. The method applies in both two
and three dimensions and for arbitrary polynomial orders. Due to space
limitations, we restrict the development to conforming and discontinuous
Galerkin (DG) methods in two and three dimensions. In the DG case, the recovered flux coincides with
state-of-the-art locally conservative flux reconstructions.

In EARM, the recovered flux is constructed as the sum of two components: an
averaged numerical flux obtained explicitly from the discrete finite element
solution, and a correction flux. The correction flux is computed by solving
a variational problem whose right-hand side is given by a residual operator
associated with the averaged numerical flux. For this reason, we call the
approach the \emph{Equilibrated Averaging Residual Method} (EARM), to
distinguish it from the classical equilibrated residual method of
\cite{ainsworth2000posteriori}.

It is important to note that the variational problem underlying EARM for solving the correction flux admits infinitely many solutions, since its null space is precisely the space of divergence-free fluxes. Our goal is to identify and solve one flux solution that is equilibrated and sufficiently accurate to guarantee the robust local efficiency in the a posteriori error analysis, while also being computationally efficient to compute.

Discontinuous Galerkin (DG) methods are well known for inherently ensuring local equilibria. We refer to \cite{Ai:07b, ern2007accurate} and references therein for earlier works on explicit conservative flux reconstructions of linear and higher-order DG elements in elliptic problems. Notably, one obvious solution to the EARM for DG solutions naturally coincides with the results in \cite{Ai:07b} and \cite{ern2007accurate} for linear and arbitrary-order DG finite element solutions.

Many researchers have studied conservative flux recovery methods for
conforming finite element methods (FEM); see, for example,
\cite{ladeveze1983error,demkowicz1987adaptive,oden1989toward,
destuynder1999explicit,ainsworth2000posteriori,AiOd:93,larson2004conservative,
vejchodsky2006guaranteed,braess2007finite,braess2008equilibrated,BrPiSc:09,
Ve:09,CaZh:11,becker2016local,odsaeter2017postprocessing,CaCaZh:20,ErnVo2015}.
Partition of unity (POU) techniques are commonly employed for localization,
particularly in the absence of explicit construction techniques.

Using a POU argument, Ladev\`eze and Leguillon \cite{ladeveze1983error}
initiated a local procedure that reduces the construction of an equilibrated
flux to vertex patch--based local problems. For continuous linear finite
element approximations of the Poisson equation in two dimensions, an
equilibrated flux in the lowest-order Raviart--Thomas space was explicitly
constructed in \cite{braess2007finite,braess2008equilibrated}. Without
introducing a constraint minimization step (see \cite{CaZh:11}), however,
this explicit construction does not yield a robust a~posteriori error
estimator with respect to jumps in the diffusion coefficient. The required
constraint minimization problem on each vertex patch can be addressed by
first computing an equilibrated flux and then applying a divergence-free
correction. For recent developments along these lines, we refer to
\cite{CaCaZh:20} and the references therein.

In \cite{ErnVo2015}, a unified approach, also based on POU techniques, was
developed. This method requires solving local mixed problems on vertex
patches associated with each mesh vertex. In contrast, the approaches in
\cite{larson2004conservative,odsaeter2017postprocessing} compute a
conservative flux by solving a global problem posed on an enriched
piecewise-constant DG space.

The equilibrated residual method of \cite{ainsworth2000posteriori}
(Chapter~6.4) also constructs a conservative flux for conforming FEM via a
partition-of-unity (POU) localization. Rather than solving directly for the
flux, the method first computes its moments, which naturally localizes the
construction to star-patch regions. However, the local star-patch problem
associated with each interior vertex (and with certain boundary vertices)
has a one-dimensional kernel. To remove this indeterminacy, an additional
constraint is imposed, and the resulting constrained local problem is
typically solved via a Lagrange-multiplier formulation. Finally, the global
flux is assembled from the computed moments.

In principle, POU localization can be applied uniformly to a wide range of
finite element methods, since it does not rely on any special structure of
the underlying basis functions. In practice, however, it is mainly used for
continuous Galerkin (CG) discretizations, because simpler recovery
procedures are available for discontinuous Galerkin (DG) methods.
Accordingly, POU-based constructions are often relatively involved, as they
require the solution of star-patch local problems that are either
constrained \cite{CaZh:11} or posed in a mixed formulation \cite{ErnVo2015}.

For localization approaches beyond POU-based methods, we refer to
\cite{becker2016local}, which studies two-dimensional Poisson problems.
There, a unified mixed formulation is proposed for continuous,
nonconforming, and discontinuous Galerkin methods in two dimensions. By
replacing the exact facet integrals with an inexact Gauss--Lobatto
quadrature, the construction can be localized, yielding conservative fluxes
for a range of FEM discretizations.
Extending this approach to three dimensions in the conforming case is,
however, not straightforward: the constraints imposed on the skeleton space
in two dimensions do not admit an obvious analogue in three dimensions, and
an appropriate Gauss--Lobatto-type quadrature rule is not available on
tetrahedral meshes.



In this paper, we propose the Orthogonal Null-space--Eliminated Equilibrated
Averaging Residual Method (ON-EARM), for which no partition-of-unity (POU)
localization is required. To ensure uniqueness, we eliminate the null space
by restricting the correction flux to the orthogonal complement of the
divergence-free null space. We show that the space generated by facet jumps
of DG functions is contained in this orthogonal complement, which yields a
well-posed variational formulation with a unique solution.

The resulting space-restricted variational problem leads to a global linear
system whose unknowns are restricted to facet degrees of freedom;
consequently, its dimension is substantially smaller than that of the linear
system arising from the original PDE discretization. In the lowest-order
(\(0\)th-order) DG case, the resulting flux coincides with the constructions
in \cite{larson2004conservative,odsaeter2017postprocessing}.

We note that, for linear and higher-order elements in two dimensions, the
global variational formulation can be localized following
\cite{becker2016local} by replacing the exact facet integrals with an
inexact Gauss--Lobatto quadrature. In this case, the resulting recovered
flux coincides with that of \cite{becker2016local} and can be computed fully
explicitly, as demonstrated in \cite{capatina2024robust}.

 Unlike in two dimensions, where Gauss--Lobatto quadrature yields a fully
localized construction, no analogous Gauss--Lobatto-type rule is available
on tetrahedral meshes in three dimensions. While tetrahedral quadrature
rules with vertex evaluations exist, their low polynomial exactness makes
them unsuitable for comparable localization and accuracy. This limitation is
specific to tetrahedra: on hexahedral meshes, tensor-product Gauss--Lobatto
rules retain vertex evaluations and high exactness, enabling full
localization. We do not analyze these localized three-dimensional variants
here; the two-dimensional case is covered in \cite{becker2016local}, and a
corresponding hexahedral analysis can likely be developed along similar
lines.

Equilibrated a~posteriori error estimators are distinguished by the fact
that their reliability constant is exactly equal to one for the conforming
error. In this work, we preserve this property within the proposed
EARM/ON-EARM framework. Moreover, in two dimensions we establish a robust
efficiency bound for the error indicators associated with ON-EARM, where
robustness means independence with respect to jumps in the diffusion
coefficient. Together, these results provide a fully reliable and robust
a~posteriori error estimation framework suitable for adaptive finite element
methods applied to elliptic interface problems.

Numerically, we test the method on a collection of singular benchmark
problems that are representative of interface effects and geometric
singularities arising in adaptive mesh refinement procedures.

In two dimensions, we consider the Kellogg interface problem and
the L-shaped domain problem, and report results for the polynomial
degrees $k=1,2,3$.
In three dimensions, we study an L-shaped domain with a vertex
singularity, again for $k=1,2,3$.
All these examples involve strong solution singularities and are
chosen to assess the robustness of the method under challenging
conditions.
What we observe is {optimal}, and in some cases even
super-convergence behavior, across all test cases.

The remainder of this paper is organized as follows.
In \cref{sec:2}, we introduce the model problem and the notations.
In \cref{sec:3} and \cref{sec:4}, we apply EARM to discontinuous
Galerkin and conforming discretizations, respectively, and show
that the resulting reconstructions recover several existing
state-of-the-art locally conservative fluxes.
In \cref{sec:5}, we establish the reliability of the associated
a~posteriori error estimator, while \cref{sec:6} is devoted to the
efficiency analysis, including robustness with respect to coefficient
jumps.
Finally, numerical results are reported in \cref{sec:7}.

\section{Model problem}\label{sec:2}
\setcounter{equation}{0}

Let $\O$ be a bounded polygonal domain in $\mathbb{R}^d, d=2,3$, with Lipschitz 
boundary $\partial \O = \overline\Gamma_D \cup \overline \Gamma_N$, where
$ \Gamma_D \cap \Gamma_N = \emptyset$.
For simplicity, assume that
$\mbox{meas}_{d-1}(\Gamma_D) \neq 0$.
Considering the diffusion problem:
\begin{equation}\label{pde}
	-\gradt (A \nabla u)  =  f   \quad\mbox{in} \quad  \O,
\end{equation} 
with boundary conditions 
\[
	u = 0 \; \mbox{ on }  \Gamma_D \quad \mbox{and} \quad
	-A \nabla u \cdot \bn=g \;\mbox{ on } 
	\Gamma_N,
\]
where $\nabla \cdot$ and $\nabla$ are the respective divergence and gradient operators; $\bn$ is the outward unit 
vector normal to the boundary; $f \in H^{-1}(\O)$ and $g\in H^{-1/2}(\Gamma_N)$ are given scalar-valued functions; and the diffusion coefficient $A(x)$ is symmetric, positive definite, piecewise constant full tensor. 


 In this paper, we use the standard notations and definitions for the Sobolev spaces. Let
\[
	H_D^1(\O) =\left\{v \in H^1(\O) \,:\,
	v=0 \mbox{ on } \Gamma_D \right\}.
\]
Then the corresponding variational problem of (\ref{pde}) is to  find $u \in H^1_D(\O)$ such that 
\begin{equation} \label{vp}
	a(u,\,v):= (A\nabla u, \nabla v) = (f, v)_{\O}- \left<g, v\right>_{\Gamma_N},
	\quad \forall  \;v\in H_D^1(\O),
\end{equation}
where $(\cdot, \cdot)_{\omega}$ is the $L^2$ inner product on the domain $\o$. 
The subscript $\omega$ is omitted from here to thereafter when $\o=\O$. 

\subsection{Notations}
 Let $\cT_h=\{K\}$ be a finite element partition of $\O$ that is regular, and denote 
 by $h_K$ the diameter of the element $K$. 
 Denote the set of all facets of the triangulation $\cT_h$ by
  \[
 	\cE := \cE_I \cup \cE_D \cup \cE_N,
 \]
 where $\cE_I$ is the set of interior element facets, and $\cE_D$ and $\cE_N$ are the sets of 
boundary facets belonging to the respective $\Gamma_D$ and $\Gamma_N$.
  For each $F \in \cE$, denote by $h_F$ the length of $F$ and by
 $\bn_F$ a unit vector normal to $F$.
 Let $K_F^+$ and $K_F^-$ be the two elements sharing the common facet $F \in \cE_I$ 
 such that the unit outward normal of $K_F^-$ coincides with $\bn_F$. When $F \in \cE_D \cup \cE_N $,
 $\bn_F$ is the unit outward normal to $\partial \O$ and denote by $K_F^-$ the element having the facet $F$.
 Note here that the term facet refers to the $d-1$ dimensional entity of the mesh. In 2D, a facet is equivalent to an edge, and in 3D, it is equivalent to a face. We also denote by $\cN$ the set of all vertices and by $\cN_{I}\subset \cN$ the set of interior vertices.
 
For each $K \in \cT_h (F \in \cE)$, let $\mathbb{P}_k(K) (\mathbb{P}_k(F))$ denote the space of polynomials of degree at most $k$ on $K (F)$. 
For each $K \in \cT_{h}$, we define a sign function $\mbox{sign}_{K}(F)$ on $\cE_{K} := \{F:  F \in \cE,
\mbox{ and } F \subset \partial K\}$:
\[
\mbox{sign}_{K}(F) = 
\begin{cases}
1 & \mbox{ if } \bn_{F} = \bn_{K}|_{F},\\
-1 & \mbox{ if } \bn_{F} =-\bn_{K}|_{F}.
\end{cases}
\]

\section{The EARM and Its Application to DG FEM}\label{sec:3}
In this section, we introduce the EARM and apply it to the DG finite element method.
The resulting equilibrated fluxes for DG solutions recover existing state-of-the-art constructions
in certain special cases, while employing modified weights to ensure robustness.

Define
\[
DG(\cT_h,k) = \{v\in L^2(\O)\,:\,v|_K\in \mathbb{P}_k(K),\quad\forall\,\,K\in\cT_h\}.
\]
Denote the $H(\mbox{div}; \O)$ conforming Raviart-Thomas (RT)  space of order $k$ with respect to $\cT_h$
by
\[
RT(\cT_h,k) = \left\{ \btau \in H(\mbox{div};\O) \,: \, \btau|_K \in RT(K,k),  \;\forall\, K\in\cT_h \right\},
\]
where $RT(K,k) = \mathbb{P}_{k}(K)^d + \bx \, \mathbb{P}_{k}(K) $.
Also let
{{
\[
RT_f(\cT_h,k) = \left\{ \btau \in RT(\cT_h,k): \gradt \btau =f_{k} \,\mbox{in} \, \O \mbox{ and} \, \btau \cdot \bn_F=
  g_{k,F} \, \mbox{on} \,F \in \cE_N \right\},
\]}}
where $f_k$ is the $L^2$ projection of $f$ onto $DG(\cT_h,k)$ and $g_{k,F}$ is the $L^2$ projection of $g|_F$ onto $\mathbb{P}_k(F)$.

  For each $F \in \cE_I$, we define the following weights:
  $\omega_F^\pm = \dfrac{ \a_{F}^\mp}{\a_F^-+\a_F^+}$ 
where $\a_F^\pm = \lambda(A|_{K_F^\pm})$ and  $\lambda(M)$ is the maximal eigenvalue of the matrix $M$. We also define  the following weighted average and jump operators:
\[
	\{v\}_w^F = \begin{cases} w_F^+ v_F^+ + w_F^- v_F^-, & F \in \cE_I,\\
	v ,& F \in \cE_D \cup \cE_N,
	\end{cases}
	\{v\}_F^w = \begin{cases} w_F^- v_F^+ + w_F^+ v_F^-, & F \in \cE_I,\\
	0,& F \in \cE_D \cup \cE_N,
	\end{cases}
 \]
 \[
		\jump{v}|_F = \left\{
	\begin{array}{ll}
		v|_{F}^- - v|_{F}^+,&\forall \,F \in \cE_I,\\[2mm]
		v|_F^-, & \forall \, F \in \cE_D \cup \cE_N.
	\end{array}
	\right.
 \]
The weighted average defined above is necessary to ensure the robustness of the error estimation with respect to the jump of $A$, see \cite{cai2017residual}. 
It is easy to show that for any $F \in \cE_{I}$,
\begin{equation} \label{weight1}
w_F^{\pm} \sqrt{\a_F^{\pm}} \le \sqrt{\a_{F,min}}, \quad\,\,
 \frac{\omega_F^+}{\sqrt{\alpha_F^-}} \le \sqrt{\frac{1}{\alpha_{F,max}}},
	\,\quad \mbox{and} \quad\,
	\frac{\omega_F^-}{\sqrt{\alpha_F^+}} \le \sqrt{\frac{1}{\alpha_{F,max}}},
\end{equation}
where $\a_{F,min} = \min(\a_{F}^{+}, \a_{F}^{-})$ and $\a_{F,max} = \max(\a_{F}^{+}, \a_{F}^{-})$.

We will also use the following commonly used identity:
 \begin{equation}\label{jump-id}
 \jump{ u v}_F = \{v\}^w_F\, \jump{u}_F + \{u\}_w^F\, \jump{v}_F, \quad \forall F \in \cE.
 \eeq

  We first present the DG variational formulation for (\ref{pde}): find $u\in V^{1+\epsilon}(\cT_h)$ with $\epsilon >0$ and 
\[ V^{s}(\cT_h) = \{v \in L^2(\O), v|_K \in H^s(K) \; \mbox{and} \;
 \gradt A \nabla v \in L^2(K), \forall K \in \cT_h\}
 \]
 such that
 \begin{equation}\label{DGV}
 a_{dg}(u,\,v) = (f,\,v)   -
\left< g, v \right>_{\Gamma_N}, \quad\forall\,\, v \in V^{1+\epsilon}(\cT_h), 
\end{equation}
where 
\begin{equation}
	\begin{split}
a_{dg}(u,v)&=(A\nabla_h u,\nabla_h v)
 +\sum_{F\in\cE \setminus \cE_N}\int_F\gamma  \dfrac{\a_{F,min}}{h_F} \jump{u}
 \jump{v}\,ds \\
& -\sum_{F\in\cE  \setminus \cE_N}\int_F\{A\nabla
u\cdot\bn_F\}_{w}^F \jump{v}ds 
 + \delta
\sum_{F\in\cE  \setminus \cE_N}\int_F\{A\nabla
v\cdot\bn_F\}_{w}^F \jump{u}ds .
\end{split}
\end{equation}
Here,  $\nabla_h$ is the discrete gradient operator defined elementwisely, 
and $\gamma$ is a positive constant, $\delta$ is a constant that takes value $1, 0$ or $-1$. 

The DG solution is to
seek $u^{dg}_k \in DG(\cT_h,k)$ such that
{{\begin{equation}\label{problem_dg}
a_{dg}(u^{dg}_k,\, v) = (f,\,v)_\O-  \left< g, v \right>_{\Gamma_N}\quad \forall\, v\in DG(\cT_h,k).
\end{equation}}}


Note that the numerical flux, $-A \nabla u_k^{dg}$, does not necessarily belong to the space $H(\mbox{div}; \O)$. 
In the first step of EARM, we define a functional for any $u_h \in DG(\cT_h, k)$:
\[
\tilde \bsigma_s: u_h \in DG(\cT_h, k) \rightarrow  \tilde \bsigma_s(u_h) \in RT(\cT_h,s), \quad
0 \le s \le k,
\]
  such that
 {\begin{equation} \label{rt:1:a}
 \int_F \tilde \bsigma_{s}(u_h) \cdot \bn_F \phi \,ds = 
 \begin{cases}
 - \int_F\{A \nabla u_h \cdot \bn_F\}_w^F \phi \,ds& \forall F \in \cE \setminus \cE_N,\\
\int g \phi \,ds, & \forall F \in \cE_N,
\end{cases}
\quad \forall \phi \in \mathbb{P}_s(F).
 \end{equation}}
 and when $s\ge1$, additionly satisfies that
\begin{equation}\label{rt:1:dg}
 \begin{split}
 	&( \tilde\bsigma_{s}(u_h) , \bm{\psi})_K =-(A \nabla u_h , \bm{\psi} )_K, \quad\forall K \in \cT_h, \forall  \bm{\psi} \in \mathbb{P}_{s-1}(K)^d.
\end{split}
 \end{equation}
 Since the recovered flux $\tilde \bsigma_{s}(u_{h})$ uses the averaging flux on the facets, we will refer to this flux as the \textit{averaging numerical flux} of $u_{h}$. 
 
  For simplicity, we assume that $g|_{F} = g_{k-1,F}$. Since $\{A \nabla u_k^{cg} \cdot \bn_F\}_w^F  \in \mathbb{P}_{k-1}(F)$, it is easy to see that when $s \ge k-1$, there holds
  {\begin{equation} \label{rt:1:aa}
 \tilde \bsigma_{s}(u_h) \cdot \bn_F = 
 \begin{cases}
 - \{A \nabla u_h \cdot \bn_F\}_w^F& \forall F \in \cE \setminus \cE_N,\\
 g , & \forall F \in \cE_N.
\end{cases}
 \end{equation}}

 

We note that the weighted averaging flux $\tilde\bsigma_s^{dg}:=\tilde \bsigma_{s}(u_{k}^{dg}) \in H(\mbox{div};\O)$, however, is not necessarily locally conservative.

In the second step of EARM, we aim to find a correction flux $ \bsigma_s^{\Delta} \in RT(\cT_h, s)$ such that 
\begin{equation}\label{DG-final-}
\hat\bsigma_s^{dg}: = \tilde\bsigma_{s}^{dg}+  \bsigma_s^{\Delta} \in RT_f(\cT_h,s).
\end{equation}

Define the \textit{averaging residual operator} for any $u_{h}$ being a finite element solution:
\begin{equation}\label{r(v)}
r_{s}(v) = \sum_{K}r_{s,K}(v), \quad \mbox{where } r_{s,K}(v) := (f - \gradt \tilde \bsigma_s(u_{h}),v)_K.
\end{equation}
Our goal is to find  $ \bsigma_s^{\Delta} \in RT(\cT_h, s)$ for $u_h$ such that
\begin{equation}\label{-equation}
  (\gradt \bsigma_{s}^{\Delta},v) =  r_{s}(v) \quad \forall v \in DG(\cT_{h},s).
\end{equation}

We now set to find a $\bsigma_s^{\Delta} \in  RT(\cT_h, s)$ specifically for the DG finite element solution $u_{k}^{dg}$.
 By the integration by parts, the definition of $ \tilde \bsigma_s^{dg}$ and \cref{problem_dg}, for any $v \in \mathbb{P}_s(K)$,  there holds
\begin{equation}\label{2.8}
\begin{split}
&r_{s}(v) =  (f,v)_{K}  + ( \tilde \bsigma_s^{dg}, \nabla v)_{K} - < \tilde \bsigma_s^{dg} \cdot \bn_{K}, v>_{\partial K}\\
=&  (f,v)_K -(A \nabla u_k^{dg}, \nabla v)_K  +
\!\!\!\sum_{F\in\cE_K \setminus \cE_N}\!\!\!\int_F \{A\nabla u_k^{dg}\cdot\bn_F\}_{w}^F\jump{v}ds 
 -
\!\!\!\sum_{F\in\cE_K \cap \cE_N} \!\!\!\left< g,v\right>_F\\
=&
\sum_{F\in\cE_K \setminus \cE_N}\int_F\gamma  \dfrac{\a_{F,min}}{h_F} \jump{u_k^{dg}}
 \jump{v}\,ds
 + \delta
\sum_{F\in\cE_{K}  \setminus \cE_N}\int_F\{A\nabla
v\cdot\bn_F\}_{w}^F \jump{u_k^{dg}}ds.
\end{split}
\end{equation}
On the other side, from \cref{-equation} and integration by parts, we also have
{{\begin{equation}\label{2.9}
\begin{split}
r_{s}(v)= r_{s,K}(v) = (\gradt \bsigma_s^{\Delta},v)_K&= -( \bsigma_s^{\Delta}, \nabla v)_K
+\sum_{F\in\cE_K}\int_F \bsigma_s^{\Delta} \cdot \bn_F \jump{v}ds.
\end{split}
\end{equation}}}
Given \cref{2.8} with (\ref{2.9}), it is natural to match and define  $\bsigma_s^{\Delta} \in RT(\cT_{h},s)$  such that
 \begin{equation} \label{rt:-cg}
\int_F \bsigma_s^{\Delta} \cdot \bn_F \phi \,ds = 
 \begin{cases}
\displaystyle  \gamma \dfrac{\a_{F,min}}{h_F}\int_F \jump{u_k^{dg}} \phi \,ds,& \forall F \in \cE \setminus \cE_N,  \quad \forall \phi \in \mathbb{P}_k(F).\\
 0, & \forall F \in  \cE_N ,
\end{cases}
 \end{equation}
 and, when, $s \ge 1$,
  \begin{equation}\label{rt::dg-cg}
 \begin{split}
 	&( \bsigma_s^{\Delta} , \bm{\psi})_K =-
	 \delta 
\sum_{F\in\cE_K  \setminus \cE_N}\int_F\{A\bm{\psi}\cdot\bn_F\}_{w}^F \jump{u_k^{dg}}ds.
	\quad \forall  \bm{\psi} \in \mathbb{P}_{s-1}(K)^d.
\end{split}
 \end{equation}

\begin{lemma}
The recovered flux $\hat \bsigma_{s}^{dg}$ defined in \cref{DG-final-} where $\bsigma_{s}^{\Delta }$ is defined in \cref{rt:-cg}-\cref{rt::dg-cg} belongs to $RT(\cT_{h},s)$ for any $ (0 \le s \le k)$. Futhermore, it is locally conservative, satisfying $ \hat \bsigma_s^{dg} \in RT_f(\cT_h,s)$. 
 \end{lemma}
 \begin{proof}
First $\hat \bsigma_{s}^{dg} \in RT(\cT_{h},s)$ is immediate.
To prove that $\hat \bsigma_s^{dg} \in RT_f(\cT_h,s)$, 
 by \cref{DG-final-}, integration by parts, the definition of  $ \bsigma_s^{\Delta}$,  \cref{2.8} and \cref{r(v)}, we have for any $v \in DG(K,s)$:
	\begin{equation*}
	\begin{split}
		&( \nabla \cdot \hat\bsigma_s^{dg}, v)_{K} 
		= 
		(\nabla \cdot  \tilde\bsigma_s^{dg}, v)_{K} +
		(\nabla \cdot   \bsigma_s^{\Delta}, v)_{K} \\
		=& 
		(\nabla \cdot  \tilde\bsigma_s^{dg}, v)_{K} -
		( \bsigma_s^{\Delta},  \nabla v)_{K} +
		\sum_{F \in \cE_{K}} <  \bsigma_s^{\Delta} \cdot \bn_{F}, \jump{v}>_{F} \\
		=&
		(\nabla \cdot  \tilde\bsigma_s^{dg}, v)_{K} +
	\sum_{F\in\cE_K  \setminus \cE_N}\delta\int_F \{A\nabla v \cdot\bn_F\}_{w}^F \jump{u_{k}^{dg}}ds 
	+
		\sum_{F \in \cE_{K} \setminus \cE_{N}} \gamma \dfrac{\a_{F,min}}{h_F}\int_F  \jump{u_k^{dg}} \jump{v} \,ds  \\
		 =& (\nabla \cdot  \tilde\bsigma_s^{dg}, v)_{K} + r_{s}(v)  = 
		(\nabla \cdot  \tilde\bsigma_s^{dg}, v)_{K} + (f - \gradt \tilde \bsigma_s^{dg},v)_K
		= (f,v)_{K}.
		\end{split}
	\end{equation*}
This completes the proof of the lemma.
 \end{proof}
The recovered flux here is similar to those introduced in \cite{Ai:07b,ern2007accurate} for $s =k$ but with modified weight. 

Since \( r_{s}(v) \) in the right-hand side of \cref{r(v)} is the residual based on the weighted averaging , we refer to our method as the \textit{Equilibrated Averaging Residual Method (EARM)} to distinguish it from the classical equilibrated residual method introduced in \cite{ainsworth2000posteriori}. 


\subsection{Algorithm for the EARM}
We now outline the pseudo-algorithm for the EARM.
Define
\[
RT_f(\cT_h,s',s) = \left\{ \btau \in RT(\cT_h,s'): \Pi_{s}(\gradt \btau) =f_{s} \,\mbox{in} \, \O \mbox{ and} \, \btau \cdot \bn_F=
  g_{s,F} \, \mbox{on} \,F \in \cE_N \right\}.
\]
Here, $\Pi_s$ denotes the $L^2$-orthogonal projection onto the space
$DG(\cT_h,s)$.
We summarize the seudo-algorithm of EARM in \cref{algorithm}.


%
%
%

\begin{algorithm}
\caption{The pseudocode for EARM}\label{algorithm}
\begin{algorithmic}[1]
\State \textbf{Input:} Finite element solution $u_{h} \in DG(\cT_{h},k)$.
\State Step 1: Compute the weighted averaging flux by \cref{rt:1:a}--\cref{rt:1:dg}:
$$
\tilde{\boldsymbol{\sigma}}_{k-1}(u_{h}) \in RT(\mathcal{T}_{h},k-1)
\quad \mbox{or} \quad
\tilde{\boldsymbol{\sigma}}_{s}(u_{h}), \quad 0 \le s \le k-1.
 $$ 
\State Step 2: Solve for a correction  flux: 
$\boldsymbol{\sigma}_{s}^{\Delta} \in RT(\mathcal{T}_{h},s)$ such that   
\begin{equation}\label{flux-equation-general}
  (\gradt \bsigma_{s}^{\Delta},v)_K =  r_{s}(v) \quad \forall v \in  DG(\cT_{h},s),
\end{equation}
where $ r_{s}(v) = (f - \gradt \tilde \bsigma_s(u_{h}),v)_K$.
\State Step 3: Obtain the equilibrated flux:
\[
	\hat \bsigma_{h} = \tilde{\boldsymbol{\sigma}}_{k-1}(u_{h}) +  \bsigma_{s}^{\Delta} \in 
	RT_f(\cT_h, max(s,k-1),s), 
\]
or
\[
	\hat \bsigma_{h} = \tilde{\boldsymbol{\sigma}}_{s}(u_{h}) +  \bsigma_{s}^{\Delta} \in 
	RT_f(\cT_h, s,s). 
\]
\end{algorithmic}
\end{algorithm}



\section{ON-EARM for Conforming FEM}\label{sec:4}

Define the $k$-th order ($k \ge 1$) conforming finite element spaces by
\[
CG(\cT_h,k) = \{v\in H^1(\O)\,:\,v|_K\in \mathbb{P}_k(K),\quad\forall\,\,K\in\cT_h\}
\]
and
\[
CG_{0,\Gamma_D}(\cT_h,k) = H_{0,\Gamma_D}^1(\O) \cap CG(\cT_h,k).
\]
The conforming finite element solution of order $k$ is to find $u_k^{cg} \in CG_{0,\Gamma_D}(\cT_h,k)$ such that
{{\begin{equation} \label{CG-solution}
	(A \nabla u_k^{cg}, \nabla v)_{\O}
	=(f,v)_\O-\left<g, v\right>_{\Gamma_N}, \quad \forall \, v  \in CG_{0,\Gamma_D}(\cT_h,k).
\end{equation} }}

To compute the correction flux in \cref{flux-equation-general} for
$u_k^{cg}$, a more careful treatment is required than in the DG case. As in the conforming
setting, a locally conservative flux does not admit a straightforward
closed-form expression.

We start by first checking the compatibility of the variational problem. Note that for any $v \in CG_{0,\Gamma_D}(\cT_h,k-1)$, by the definition of $ \tilde \bsigma_{k-1}^{cg}:= \tilde \bsigma_{k-1} (u_k^{cg})$ and integration by parts, there holds
 \begin{equation}
 \begin{split}
r(v): =&(f,v) - (\gradt \tilde \bsigma_{k-1}^{cg},v) =
    (f,v)  -  (A \nabla u_k^{cg}, \nabla v) -
    < g, v>_{\Gamma_{N}}=0.
    \end{split}
 \end{equation}
 
 Specifically, for any $v \in CG_{0,\Gamma_D}(\cT_h,s), s \le k-1$ ,  and $\mbox{supp}(v) = K$ for any $K \in \cT_{h}$, we also have $r(v) =0.$
The residual problem \cref{flux-equation-general} then should satisfy
\begin{equation}\label{compatbility-interior}
0 = r(v) =(\gradt \bsigma_{s}^{\Delta},v)_K
=
 -( \bsigma_s^{\Delta}, \nabla v)_K =0.
\end{equation}

To ensure the compatibility in \cref{compatbility-interior}, it is natural to impose the following restriction for the interior degree of freedom of $\bsigma_s^{\Delta}$ by
\begin{equation}\label{cg-interior-dof}
( \bsigma_s^{\Delta},  \bm{\psi})_K =0 \quad \forall   \bm{\psi} \in \mathbb{P}_{s-1}(K)^d \quad \mbox{when } s\ge 1.
\end{equation}

Define the space for $s \ge 1$,
 \[
 \mathring{RT}(\cT_h,s) = \left\{ \btau\in RT(\cT_h,s) : ( \btau,  \bm{\psi})_K =0 \; \forall  \bm{\psi} \in \mathbb{P}_{s-1}(K)^d, K \in \cT_h, \btau \cdot \bn_F=
0,F \in \cE_{N} \right\}.
 \]
 In the case $s=0$, we let $ \mathring{RT}(\cT_h,0) =  {RT}(\cT_h,0)$.
 Then,  \cref{flux-equation-general}, becomes finding $ \bsigma_s^{\Delta} \in \mathring{RT}(\cT_{h},s)$ such that
 \begin{equation}\label{flux-cg-problem}
\begin{split}
\sum_{F\in\cE \setminus \cE_{N}}\int_F \bsigma_s^{\Delta} \cdot \bn_F \jump{v}ds =r(v)
\quad \forall v \in DG(\cT_{h},s).
\end{split}
\end{equation}
 
 We have the following compatibility result for \cref{flux-cg-problem}.
 \begin{lemma}
 Let $\mathcal{A}(\bftau_h,v) := \underset{F \in \cE \setminus \cE_{N}}{ \sum } \left<\bftau_h \cdot \bn_F, \jump{v} \right>_F$ be a bilinear form defined on $\bftau_h \in \mathring{RT}(\cT_h, s) \times DG(\cT_{h},s), \,  0 \le s \le k-1$.
There holds
 \begin{equation}\label{compatibility-cg}
\begin{split}
	&\mathcal{A}(\bftau_h,v)= r(v) = 0, \quad\forall v \in CG_{0,\Gamma_D}(\cT_h,s). \quad
\end{split}
\end{equation}
where $r(v) = (f - \nabla \cdot \tilde\bsigma_{k-1}^{cg}, v)_{\O}$.
 \end{lemma}
 \begin{proof}
 The proof is immediate.
 \end{proof}

 
\subsection{Null-Space Elimination for $\bsigma_{s}^{\Delta}$}
\label{subsec:restriction}
Note that \cref{flux-cg-problem} has infinitely many solutions given the kernel of divergence-free space.
In this subsection, we derive a recovery method that removes the null-space and restricts the correction flux  in the orthogonal complement of divergence-free null-space. 

Denote the divergence-free space in $\mathring{RT}(\cT_{h},s)$ as
 \[\mathring{RT}^{0}(\cT_{h},s) = \{\bftau \in \mathring{RT}(\cT_{h},s): \nabla \cdot \bftau =0  \}.\]
Also denote by $\mathring{RT}^{0}(\cT_{h},s)^{\perp}$ the orthogonal complement of the divergence free space $\mathring{RT}^{0}(\cT_{h},s)$.

%

  Now  for any $w \in DG(\cT_{h},s)$, define a mapping $\mathbf{S}:  w \in DG(\cT_{h},s) \to$
  $\mathbf{S}(w) \in \mathring{RT}(\cT_{h},s)$ such that
  \beq \label{dg-}
  \mathbf{S}(w) \cdot \bn_{F}|_{F} = A_{F} h_{F}^{-1}\jump{w}|_F \quad \forall F \in {\cE \setminus \cE_{N}},
  \eeq
  where $A_{F} = \min(\a_{F}^{+}, \a_{F}^{-})$ when $F \in \cE_{I}$ and $A_{F} =\a_{F}^{-}$ when $F \in \cE_{D}$.
\begin{lemma}\label{lem:perp-space}
The following relationship holds:
 \[
 	\{ \mathbf{S}(w): w \in DG(\cT_{h},s) \} \subset \mathring{RT}^{0}(\cT_{h},s)^{\perp}.
	\]
\end{lemma}
\begin{proof}
We first observe that for any $\bftau_h \in \mathring{RT}^{0}(\cT_{h},s)$, there holds
 \begin{equation}\label{kernel3}
0 = (\nabla  \cdot \bftau_h,  v)_{\cT_{h}}  \Leftrightarrow
 \underset{F \in \cE\setminus \cE_{N}}{ \sum } \left<\bftau_h \cdot \bn_F, \jump{v} \right>_F =\mathcal{A}(\bftau_h,v) =0 \quad \forall v \in DG(\cT_{h},s).
 \end{equation}
 Thus, $\bftau_h \in  \mathring{RT}^{0}(\cT_{h},s)^{\perp}$ if and only if
    \begin{equation}\label{sufficientCondition}
\mathcal{A}(\bftau_h,v) =0 \quad \forall v \in DG(\cT_{h},s) \Rightarrow \bftau_h \equiv 0.
 \end{equation}
Therefore, to prove that $\mathbf{S}(w) \in  \mathring{RT}^{0}(\cT_{h},s)^{\perp}$, it is sufficient to prove that $\mathbf{S}(w)$ satisfies \cref{sufficientCondition} for all $w \in DG(\cT_{h},s)$.
Assuming
   \begin{equation}
   \mathcal{A}(\mathbf{S}(w),v) =0 \quad \forall v \in DG(\cT_{h},s),
 \end{equation}
immediately yields
    \begin{equation}
   \mathcal{A}(\mathbf{S}(w),w)  =0,
 \end{equation}
i.e., $\|A_{F}^{1/2} h_{F}^{-1/2}\jump{w}\|_{\cE\setminus \cE_{N}} = 0$, hence, $\| \mathbf{S}({w}) \cdot \bn_{F} \|_{\cE\setminus \cE_{N}} \equiv 0$. This completes the proof of the lemma.
\end{proof}

Thanks to \cref{lem:perp-space}, we reformulate the problem \cref{flux-cg-problem} by restricting the flux  in the space of $\mathring{RT}^{0}(\cT_{h},s)^{\perp}$   as following:
finding $u_s^{\Delta} \in DG(\cT_h, s) $ such that
\begin{equation}\label{b-correction}
	\mathcal{A}( \jump{u_s^{\Delta}}, \jump{v})
= r(v) \quad \forall \, v \in DG(\cT_h,s),
\end{equation}
where
\[
\mathcal{A}( \jump{u_s^{\Delta}}, \jump{v}) 
:=\sum_{F \in \cE \setminus \cE_{N}} \int_{F}  A_{F}h_{F}^{-1} \jump{u_s^{\Delta}} \jump{v} \,ds.
\]

Note that only the information of $\jump{u_{s}^{\Delta}}$ on facets $\cE \setminus \cE_{N}$ is used in the formulation. We further
define for any $s \ge 1$ the quotient spaces:
\[
	 DG^{0}(\cT_h, s) = \{v \in DG(\cT_h, s): v|F=0 \, \forall F \in \cE_{N} \} / \{v \in  CG(\cT_h, s)\}.
\]
Here we have used $A/B$ to denote the quotient space of $A$ by $B$.
 \cref{b-correction} is then equivalent to finding $u_s^{\Delta} \in DG^{0}(\cT_h, s) $ such that
\begin{equation}\label{b-correction-A} 
	\mathcal{A}( \jump{u_s^{\Delta}}, \jump{v})_{F}
= r(v) \quad \forall \, v \in DG^{0}(\cT_h,s).
\end{equation}

One feature of \cref{b-correction-A} is the well-posedness thanks to the same trial and test spaces along with continuity and coercivity in the space of $DG^{0}(\cT_h,s)$.

\begin{lemma}
	\cref{b-correction-A} has a unique solution $u_s^{\Delta} \in DG^{0}(\cT_h, s)$ for all integers $0 \le s \le k-1$.
\end{lemma}
\begin{proof}
We first have that the bilinear form $	\mathcal{A}( \jump{u_s^{\Delta}}, \jump{v})_{F}$ is continuous and coercive in the space of $DG^{0}(\cT_h, s)$ under the norm:
\beq \label{tri-norm}
\tri v \tri =  \sqrt{\sum_{F \in \cE\setminus \cE_{N}} h_{F}^{-1}\| A_{F}^{1/2}\jump{v}\|_{F}^{2}}.
\eeq Then, by the Lax-Milgram theorem, \cref{b-correction-A}  has a unique solution for all $0 \le s \le k-1$. 
\end{proof}
\begin{remark}
When $s=0,$ our method coincides with the equilibrated flux recovered in \cite{odsaeter2017postprocessing}.
	Though the solution of \cref{b-correction-A} is unique, it renders a global problem. However, the number of degrees of freedom for the global problem is much less than the original problem, as it includes only unknowns restricted to facet degrees of freedom.
	
Moreover, in two dimensions, for \(1 \leq s \leq k-1\), we can localize it using the same techniques as in \cite{becker2016local} by replacing the exact integrals, \(\int_{F} A_{F}h_{F}^{-1} \jump{u_s^{\Delta}} \jump{v} \,ds\), with an inexact Gauss-Lobatto quadrature. 
With Gauss-Lobatto quadrature in two dimensions, our recovered flux coincides with that in \cite{becker2016local}. Furthermore, the recovered flux can be computed entirely explicitly, as shown in \cite{capatina2024robust}.

\end{remark}

\begin{lemma}\label{lem:sigma-hat-1}
Define for some $0 \le s \le k-1$
\begin{equation}\label{recover--cg}
	\hat \bsigma_h^{cg} = \tilde \bsigma_{k-1}^{cg} + \bsigma_s^{\Delta},
\end{equation}
where $\bsigma_s^{\Delta} =\mathbf{S}({u_{s}^{\Delta}})$ in which $u_{s}^{\Delta}$ is the solution to \cref{b-correction-A}.
Then $\hat \bsigma_h^{cg}  \in RT(\cT_{h}, k-1, s)$ satisfies
	\begin{equation}\label{cg-conservation-1}
		(\nabla \cdot \hat \bsigma_h^{cg}, v) = (f,v) \quad \forall v \in DG(\cT_{h},s).
	\end{equation}
\end{lemma}
\begin{proof}
	It is sufficient to prove \cref{cg-conservation-1} for any $v \in \mathbb{P}_{s}(K) \subset DG(\cT_h,s)$.
	By the definitions, integration by parts, \cref{b-correction-A} 
	\begin{equation}
	\begin{split}
		&( \nabla \cdot \hat\bsigma_h^{cg}, v)_{K} 
		= 
		(\nabla \cdot  \tilde\bsigma_{k-1}^{cg}, v)_{K} +
		(\nabla \cdot   \bsigma_s^{\Delta}, v)_{K} \\
		&= 
		(\nabla \cdot  \tilde\bsigma_{k-1}^{cg}, v)_{K} +
		( \bsigma_s^{\Delta},  \nabla v)_{K} +
		(  \bsigma_s^{\Delta} \cdot \bn_{F}, \jump{v})_{\partial K} \\
		&= 
		(\nabla \cdot  \tilde\bsigma_{k-1}^{cg}, v)_{K} +
		\sum_{F \in \cE_{K} \setminus \cE_{N}}( A_{F} h_{F}^{-1} \jump{u_{s}^{\Delta}}, \jump{v})_{F} 
		=  (\nabla \cdot  \tilde\bsigma_{k-1}^{cg}, v)_{K} + r(v) \\
		&= (\nabla \cdot  \tilde\bsigma_{k-1}^{cg}, v)_{K}  +  (f,v)_K
		- (\gradt \tilde \bsigma_{k-1}^{cg},v)_{K}  =  (f,v)_{K}.
		\end{split}
	\end{equation}
	We have completed the proof of the lemma.
\end{proof}
\begin{lemma}\label{lem:sigma-hat-2}
Define for some $0 \le s \le k-1$
\begin{equation}\label{recover--cg1}
	\hat \bsigma_s^{cg} = \tilde \bsigma_{s}^{cg} + \bsigma_s^{\Delta},
\end{equation}
where $\bsigma_s^{\Delta} =\mathbf{S}({u_{s}^{\Delta}})$ in which $u_{s}^{\Delta}$ is the solution to \cref{b-correction-A}.
Then $\hat \bsigma_s^{cg}  \in RT_f (\cT_{h},s) $ satisfies
	\begin{equation}\label{convervation1}
		(\nabla \cdot \hat \bsigma_s^{cg}, v) = (f,v) \quad \forall v \in DG(\cT_{h},s).
	\end{equation}
\end{lemma}
\begin{proof}
	\cref{convervation1} can be proved similarly as in \cref{lem:sigma-hat-1} and the fact that 
	\[
	(\nabla \cdot  \tilde\bsigma_{s}^{cg}, v)_{K}
		= (\gradt \tilde \bsigma_{k-1}^{cg},v)_{K}  \quad \forall v \in DG(\cT_{h},s).
	\]
\end{proof}

\section{Automatic Global Reliability}\label{sec:5}
We first cite the following reliability result proved in \cite{cai2021generalized}.
 \begin{theorem}\label{PS:2d}
Let $u\in H^1_{D}(\O)$ be the solution of  {\em (\ref{pde})}. In two and three dimensions,
 for all $w \in H^1(\cT_{h})$, we have
\[
	 \|A^{1/2}\nabla_h(u-w)\|^2 =
	 \inf_{\btau\in \S_f(\O)}\| A^{-1/2}\btau+
	 A^{1/2}\nabla_h w\|^2 + \inf_{v\in H^1_{D}(\O)}\|A^{1/2}\nabla_h(v-w)\|^2.
\]
 \end{theorem}
 where
  \[
\S_f(\O) = \Big\{ \btau \in H(\divvr;\O) : \gradt \btau =f  \mbox{ in } \O \;\mbox{ and }
  \; \btau \cdot \bn = {g  \mbox{ on } \Gamma_N} \Big\}.
  \]
  Based on the theorem and Corollary~3.5 in \cite{cai2021generalized}, 
the construction of an equilibrated a posteriori
error estimator for discontinuous finite element solutions is reduced to recover an equilibrated 
in $\S_f(\O)$ and to recover 
either a potential function in $H^1(\Omega)$ or a curl free vector-valued
function in $H(\curll;\Omega)$. In this paper, we focus on the part of equilibrated flux recovery.
We note that $  \displaystyle\inf_{\btau\in \S_f(\O)}\| A^{-1/2}\btau+A^{1/2}\nabla_h w\|^2 $ is usually referred to as the conforming error of $w$.

Note that in our case, the recovered flux lies in $RT_f(\cT_h,s)$ or $RT_f(\cT_h, k-1, s)$, which is not generally in $\S_f(\O)$. We can adjust and obtain the following reliability result similar to \cref{PS:2d} except with an additional oscillation term on the right, 
\begin{equation}\label{reliability-result}
\begin{split}
	 \|A^{1/2}\nabla_h(u- u_{k}^{cg})\|^2 &\le 
	\| A^{-1/2}\hat \bsigma_{h} +
	 A^{1/2}\nabla_h u_{k}^{cg}\|^2 \\
	 & + \inf_{v\in H^1_{D}(\O)}\|A^{1/2}\nabla_h(v-u_{k}^{cg})\|^2 + c\|f - f_{s}\|^{2}_{\O}.
\end{split}
\end{equation}
This additional oscillation term is of higher order when $f|_{K}$ is smooth for all $K \in \cT_{h}$. It is usually neglected in the computation.

We define the local indicator for the conforming error by
\begin{equation} \label{estimators:cf_l}
	 \eta_{\sigma,K} = \|A^{-1/2} (\hat \bsigma_h - \bsigma_{h}) \|_K  
	\end{equation}
where 
$ \bsigma_h = - A\nabla u_{h}$ is the numerical flux corresponding to the finite element solution $u_{h}$ and
$\hat \bsigma_h$ is a recovered equilibrated flux based on $\bsigma_h$.
The  global estimator is then defined as
\begin{equation} \label{estimators:cf_g}
	\eta_\sigma =\left( \sum_{K \in \cT_h} \eta^2_{\sigma,K} \right)^{1/2}
	=\|A^{-1/2} (\hat \bsigma_h-  \bsigma_h) \|_{0,\O},
\eeq 
in which $\hat \bsigma_h$ is the recovered equilibrated flux based on $u_{h}$.

\section{Robust Efficiency} \label{sec:6}
In this section, we prove robust efficiency of the error indicator
defined in \cref{estimators:cf_l} for the recovered flux
$\hat\bsigma_{h} =\hat{\bsigma}_{h}^{cg}$ from \cref{lem:sigma-hat-1},
when applied to the conforming finite element solution $u_k^{cg}$.
The remaining cases for  $\hat{\bsigma}_{s}^{cg}, 0 \le s \le k-1 $ can be treated similarly.
Here, \emph{robustness} means that the efficiency constant is
independent of the jump in the diffusion coefficient $A(x)$.

 For recovered fluxes constructed from discontinuous Galerkin solutions, we refer to \cite{Ai:07b,ern2007accurate,cai2021generalized,ainsworth2000posteriori}.

%


For simplicity, we assume that the diffusion coefficient $A(x)$ is a piecewise constant function and that $A_{F}^{-} \le A_{F}^{+}$ for all $F \in \cE_{I}$. From here to thereafter, we use $a \lesssim b$ to denote that  that 
$a \le C b$ for a generic constant that is independent of the mesh size and the jump of $A$.

For simplicity, we restrict our analysis in the two dimensions.
To show that the efficiency constant for the conforming case is independent of the jump of $A(x)$,
as usual, we assume that the distribution of the coefficients $A_K$
for all $K\in \cT_{h}$ is locally quasi-monotone \cite{petzoldt2002posteriori}, which is
slightly weaker than Hypothesis 2.7 in \cite{bernardi2000adaptive}.
The assumption has been used in the literature and remains essential for rigorous theoretical proofs. However, numerical experiments and practical applications indicate that this assumption is not necessary in computations, as the method performs reliably even when it does not strictly hold, e.g., see \cite{cai2017improved}.

For convenience of readers, we restate the definition of quasi-monotonicity.
Let $\o_z$ be the union of all elements having $z$ as a vertex.
For any $z\in\cN$, let
 \[
 \hat{\omega}_z=\{K\in\omega_z \,:\, A_K = \max_{K'\in\omega_z} A_{K'}\}.
\]
\begin{definition}\label{defnquasimonotone}
Given a vertex $z \in \cN$, the distribution of the coefficients $A_K$, $K\in\omega_z$,
is said to be {\em quasi-monotone} with respect to the vertex $z$
if there exists a subset $\tilde{\o}_{K,z,qm}$ of $\omega_z$ such that the union
of elements in $\tilde{\o}_{K,z,qm}$ is a Lipschitz domain and that
\begin{itemize}
\item if $z\in\cN\backslash\cN_\sD$, then $\{K\}\cup \hat{\o}_z
\subset \tilde{\o}_{K,z,qm}$
and $A_K\leq A_{K'} \; \forall K' \in \tilde{\o}_{K,z,qm}$;
\item if $z\in\cN_\sD$, then $K\in \tilde{\o}_{K,z,qm}$,
$\p\tilde{\omega}_{K,z,qm}\cap\Gamma_D \neq \emptyset$, and
$A_K\leq A_{K'} \; \forall K' \in \tilde{\o}_{K,z,qm}$.
\end{itemize}
The distribution of the coefficients $A_K$, $K\in\cT$, is said to be
locally {\em quasi-monotone} if it is quasi-monotone with respect to
every vertex $z\in\cN$.
\end{definition}
 For a function $v \in DG(\cT_h, s)$, we choose to employ the classical Lagrangian basis functions. 
 For a function $v \in DG^{0}(\cT_h, s)$,  since only  $\jump{v}|_{\cE \setminus \cE_{N}}$ is needed, from here to thereafter, we assume that $v$ satisfies the following conditions. 
 \begin{enumerate}
 \item
The interior degrees of freedom of $v$ are zero, i.e., $v(\bx_{i,K}) =0, i=1, \cdots, m_{s}$.
 \item
Let $\{\bx_{i,F}\}_{i=1}^{s-1} (s \ge 1)$be the set of all Lagrange points on $F$ excluding the vertices $\bx_{s,F}$ and $\bx_{e,F}$ for the space $P_{k}(F)$. For each $F \in \cE_{I}$ and $\bx_{i,F} \in F$, let $v|_{K_{F}^{+}}(\bx_{i,F})=0$. 
\item
For each $z \in \cN$, we denote by $K_{z}$ be a element in $\hat w_{z}$ and let $v|_{K_{z}}(z) =0$.
\item When $F \in \cE_{N}$, we simply let $v|_{F} \equiv 0$.
 \end{enumerate}
 It's not difficult to prove that with the above four types of restrcitions, any function  $v \in DG^{0}(\cT_h, s)$ is uniquely determined. 
 
 \begin{lemma}[Uniqueness in $ DG^{0}(\cT_h, s)$]
Let $u,v \in DG^{0}(\cT_h,s)$ satisfy the above assumptions.
If
\[
\jump{u} = \jump{v}
\quad \text{on } \cE \setminus \cE_N,
\]
then
\[
u \equiv v \quad \text{in } \Omega.
\]
\end{lemma}

\begin{proof}
Let $w := u - v \in DG^{0}(\cT_h,s)$.
By the assumption $\jump{u} = \jump{v}$ on $\cE \setminus \cE_N$ and
$u|_{F} = v|_{F} \equiv 0$ for $F \in \cE_N$, we have
\[
\jump{w} = 0 \quad \text{on } \cE.
\]
Hence, $w \in CG(\cT_{h},s)$ and $w =0$ on $\partial \O$.

Moreover, by Assumption~1, $w$ vanishes at all interior degrees of
freedom. By Assumption~2, $w$ vanishes at all interior facet Lagrange
points, except possibly at the vertices. Finally, by Assumption~3,
$w$ vanishes at all vertex degrees of freedom. Hence, all degrees of
freedom of $w$ are zero, and therefore
\[
w \equiv 0.
\]
Consequently, $u \equiv v$. This completes the proof of the lemma.

\end{proof}

We first prove the following lemma.
\begin{lemma}\label{lemma:boundforH1}
For any $v \in DG^{0}(\cT_h, s)$, there holds
\begin{equation}\label{qw-0}
	\|A^{1/2} \nabla v\|_{\cT_{h}} \le C  \sum_{F \in \cE \setminus \cE_{N}} h_{F}^{-1/2} A_{F}^{1/2} \|\jump{v}\|_{F}.
\end{equation}
where the constant $C$ is independent of the mesh and the jump of the coefficient $A(x)$,
\end{lemma}
\begin{proof}
For any $K \in \cT_{h}$, we first have that
\beq\label{qw-1}
\begin{split}
	\|\nabla v\|_{K} \lesssim  \sum_{z \in \cN_{K}} | v_{K}(z)| 
	+ \sum_{F \in \cE_{K}} \sum_{i=1}^{s-1} |v_{K}(\bx_{i,F})|.
\end{split}
\eeq
Fix any $z \in \cN_K$. If $K = K_{z}$, we have $v_{K}(z)=0$.
If $K \neq K_{z}$, applying the triangle equality and the property of $v$ yields
\beq
\begin{split}
 | v_{K}(z)|  \le \sum_{F \subset \tilde \o_{K,z,qm} \cap \cE_{z}} | \jump{v}|_F(z)|.
\end{split}
\eeq
Here we choose $\tilde \o_{K,z,qm}$ to be open, so $\tilde \o_{K,z,qm} \cap \cE_{z}$ does not contain the two boundary edges $ \partial_{\tilde \o_{K,z,qm}} \cap \cE_{z}$.
Since $A(x)$ is monotone along $\tilde \o_{K,z,qm}$, we have
\beq
\begin{split}
 A_{K}^{1/2}| v_{K}(z)|  \le \sum_{F \subset \tilde \o_{K,z,qm} \cap \cE_{z}} A_{F}^{1/2}| \jump{v}|_F(z)|.
\end{split}
\eeq
Finally, for any $\bx_{i,F}$, $i=1, \cdots, s-1$ and $F \in \cE_{K}$ we have
\beq\label{qw-2}
A_{K}^{1/2}|v_{K}(\bx_{i,F})| = 
\begin{cases}
0 & \mbox{if }K = K_{F}^{+}\\
A_{K}^{1/2}| \jump{v}|_{F}(\bx_{i,F})|   & \mbox{if } K = K_{F}^{-}
\end{cases}
\le A_{F}^{1/2}| \jump{v}|_{F}(\bx_{i,F})|.
\eeq
Recall we assumed that $A_{F}^{-} \le A_{F}^{+}$. 
Combing \cref{qw-1}--\cref{qw-2}, we have for any $K \in \cT_{h}$,
\beq
\begin{split}
	\|A^{1/2}\nabla v\|_{K}& \lesssim  
	\sum_{z \in \cN_{K}}  \sum_{F \in \cE_{z}} 
	A_{F}^{1/2}| \jump{v}|_{F}(z)
	+  \sum_{F \in \cE_K}\sum_{i=1}^{s-1}A_{F}^{1/2}| \jump{v}|_{F}(\bx_{i,F})| 
	\\
	& \lesssim\sum_{z \in \cN_{K}}  \sum_{F \in \cE_{z}} A_{F}^{1/2} h_{F}^{-1/2} \| \jump{v}\|_{F}.
\end{split}
\eeq
and, hence, \cref{qw-0}. This completes the proof of the lemma.
\end{proof}


\begin{lemma}\label{lem:effi-cg-1}
Let $u_{s}^{\Delta} \in DG^{0}(\cT_h, s)$ be the solution of \cref{b-correction-A}. We have the following estimation for $\tri u_{s}^{\Delta}\tri$:
\beq \label{global-efficiency}
	\tri u_{s}^{\Delta}\tri \le C \| A^{1/2} \nabla (u - u_{k}^{cg})\| + \osc{(f)},
	\eeq
	where the constant $C$ is independent of the mesh size and the jump of the coefficient of $A(x)$. 
\end{lemma}
 \begin{proof}
 Frist, we observe that
 \begin{equation}\label{norm-bounds}
  \tri u_{s}^{\Delta} \tri = \sup_{v \in DG^{0}(\cT_h, s), \tri v\tri =1} {\mathcal{A}(u_{s}^{\Delta}, v)} =
  \sup_{v \in DG^{0}(\cT_h, s), \tri v\tri =1} r (v).
  \end{equation}
  Applying integration by parts,  the definition of $\tilde \bsigma_{k-1}^{cg}$ and \cref{jump-id},  we have
\begin{equation}\label{effi:cg-a}
\begin{split}
r(v) &
\!=\! (f - \nabla \cdot  \tilde \bsigma_{k-1}^{cg}, v)_{\O}
\!=\! (f,  v)_{\O} \!+\!  \sum_{K \in \cT_{h}} \left((\tilde \bsigma_{k-1}^{cg},  \nabla v)_{K} 
-
<\tilde \bsigma_{k-1}^{cg} \cdot \bn_{K},  v >_{\partial K} \right)
\notag\\
&= (f,  v)_{\O} - \sum_{K \in \cT_{h}}(A \nabla u_{k}^{cg},  \nabla v)_{K} 
+
\sum_{F \in \cE \setminus \cE_{N}}< \{ A \nabla u_{k}^{cg} \cdot \bn_{F}\}_{w}^{F} , \jump{ v}) >_{F}\\
\\
&= (f + \nabla \cdot A \nabla u_{k}^{cg}, v)_{\O}
-
\sum_{F \in \cE_{I}}< (\jump{ A \nabla u_{k}^{cg} \cdot \bn_{F}} ,  \{v\}_{F}^{w}) >_{F} \notag\\
\end{split}
\end{equation}

By the triangle and Cauchy Schwartz inequalities, and \cref{weight1}, we have for any $F \in \cE_{I}$,
\begin{equation}\label{bound-part2a}
\begin{split}
&< \jump{ A \nabla u_{k}^{cg} \cdot \bn_{F}} ,  \{v\}_{F}^{w}>_{F}
\lesssim
		\|  \jump{ A \nabla u_{k}^{cg} \cdot \bn_{F}}\|_{F}  
		(\omega_{F}^+\| v_{K_F^-}\|_{F}
		+\omega_{F}^-\|v_{K_F^+}\|_{F} )\\
&\lesssim
\|  A_{F,max}^{-1/2}\jump{ A \nabla u_{k}^{cg} \cdot \bn_{F}}\|_{F}  
		(\| (A^{1/2}v)_{K_F^-}\|_{F}
		+\|(A^{1/2}v)_{K_F^+}\|_{F} ).
\end{split}
\end{equation}
Applying similar techniques as in the proof of \cref{lemma:boundforH1}
we can have the following bounds
\begin{equation}\label{bound-part2b}
(\| (A^{1/2}v)_{K_F^-}\|_{F}
		+\|(A^{1/2}v)_{K_F^+}\|_{F} )
\lesssim\sum_{z \in \cN_{F}}  \sum_{F \in \cE_{z}} A_{F}^{1/2} \| \jump{v}\|_{F},
\end{equation}
and
\begin{equation}\label{bound-part2bb}
\| A^{1/2} v\|_{K} \lesssim\sum_{z \in \cN_{K}}  \sum_{F \in \cE_{z}} A_{F}^{1/2} h_{F}^{1/2} \| \jump{v}\|_{F}
\end{equation}

Finally, applying \cref{norm-bounds}--\cref{bound-part2bb} and the Cauchy Schwartz inequality, and the fact that $\tri v \tri =1$, we have 
\begin{equation}\label{effi:cg-aa}
\begin{split}
&r(v)\le 
 \sqrt{\sum_{K \in \cT_{h}}h_{K}^{2}\|A^{-1/2}( f + \nabla  \cdot A \nabla u_{k}^{cg})\|_K^{2} }
 \sqrt{\sum_{K \in \cT_{h}}h_{K}^{-2} \| A^{1/2} v\|_K^{2}} \notag\\
 &+
\sqrt{\sum_{F \in \cE \setminus \cE_{N}} h_{F} A_{F,max}^{-1}\| \jump{ A \nabla u_{k}^{cg} \cdot \bn_{F}} \|_{F}^{2}}
 \sqrt{\sum_{F \in \cE \setminus \cE_{N}} h_{F}^{-1}(\| A^{1/2}v_{K_F^-}\|_{F}^{2}
		+\|A^{1/2}v_{K_F^+}\|_{F}^{2} )}
\notag\\
&\lesssim 
\sqrt{ {\sum_{K \in \cT_{h}}h_{K}^{2}\| A^{-1/2}(f + \nabla  \cdot A \nabla u_{k}^{cg})\|_K^{2} }
 +
 {\sum_{F \in \cE \setminus \cE_{N}}  \dfrac{h_{F}}{A_{F,max}}\| \jump{ A \nabla u_{k}^{cg} \cdot \bn_{F}} \|_{F}^{2}}} \; \tri v \tri.
\notag \\
\end{split}
\end{equation}

From the classical efficiency results, we also have
 \begin{equation}\label{efficiency-ele-residual}
  h_{K}\| A^{-1/2}(f + \nabla \cdot A \nabla u_{k}^{cg}) \|_K \lesssim \| A^{1/2} \nabla (u - u_{k}^{cg})\|_{\o_{K}} + \mbox{osc}(f)
 \end{equation}
 and
 \begin{equation}\label{efficiency--jump}
 \|A_{F,max}^{-1/2} \jump{ A \nabla u_{k}^{cg} \cdot \bn_{F}} \|_{F} \lesssim
  \| A^{1/2} \nabla (u - u_{k}^{cg})\|_{\o_{F}} + \mbox{osc}(f)
 \end{equation}
 where $\o_{K}$ and $\o_{F}$ are some local neighborhood of $K$ and $F$, respectively; and the involved constant does not depend on the jump of $A$.
Combining all yields \cref{global-efficiency}. This completes the proof of the lemma.
 \end{proof}

\begin{theorem}\label{lem:effi-cg-2}
Recall the local error indicator
\begin{equation} \label{cg-local-indicator}
	 \eta_{\sigma,K} = \|A^{-1/2} \hat \bsigma_h^{cg} + A^{1/2} \nabla u_k^{cg} \|_K 
	\end{equation}
	where $ \hat\bsigma_{h}^{cg}$ is defined in \cref{lem:sigma-hat-1}.
	Then we have the following global efficiency bound:
	\begin{equation}\label{global-effi-cg}
	 \eta_{\sigma} \le
	 C  \| A^{1/2} \nabla (u - u_{k}^{cg})\|_{\cT_{h}} + \mbox{osc}(f),
\end{equation}
where the constant is independent of the mesh and the jump of the coefficient $A(x)$.
\end{theorem}
\begin{proof}
First, applying the triangle inequality,
\beq\label{tiangle-inequality}
\begin{split}
 \eta_{\sigma,K} &= \|A^{-1/2} (\hat \bsigma_h^{cg} - \tilde \bsigma_{k-1}^{cg}) \|_K 
 +
 \|A^{-1/2}\tilde \bsigma_{k-1}^{cg} + A^{1/2} \nabla u_k^{cg}  \|_K\\
 &= \|A^{-1/2} \bsigma_s^{\Delta}\|_K 
 +
 \|A^{-1/2}\tilde \bsigma_{k-1}^{cg} + A^{1/2} \nabla u_k^{cg}  \|_K.
 \end{split}
\eeq

To bound the second term, we first have by the equivalence of norms, \cref{weight1}, \cref{rt:1:aa} and \cref{efficiency--jump},
\beq\label{bound-part1}
\begin{split}
	&\| A^{-1/2}\tilde \bsigma_{k-1}^{cg} + A^{1/2} \nabla u_k^{cg}  \|_K 
	\lesssim
	\sum_{F \in \cE_{K}} h_{F}^{1/2}\| A_{K}^{-1/2}(\tilde \bsigma_{k-1}^{cg} 
	+ A_{K} \nabla u_k^{cg} ) \cdot \bn_{F}\|_{F}\\
	&
	\lesssim
	\sum_{F \in \cE_{K}}h_{F}^{1/2}\| A_{F,max}^{-1/2}\jump{A \nabla u_{k}^{cg}} \cdot \bn_{F}\|_{F}
	\lesssim
	 \sum_{F \in \cE_{K}}\| A^{1/2} \nabla (u - u_{k}^{cg})\|_{\o_{F}} + \mbox{osc}(f).
\end{split}
\eeq

It is then sufficient to prove
	\begin{equation}\label{global-effi-cg-a}
	\|  A^{-1/2}\bsigma_{s}^{\Delta} \|_{\cT_{h}} \le
	C   \| A^{1/2} \nabla (u - u_{k}^{cg})\|_{\cT_{h}} + \mbox{osc}(f).
\end{equation}

By the definition of $\bsigma_{s}^{\Delta}$,  \cref{dg-}, we have for each $K \in \cT_{h}$,
\beq
\begin{split}
	\|A^{-1/2} \bsigma_{s}^{\Delta}\|_{K} &\lesssim 
	 \sum_{F \in \cE_{K}} h_{F}^{1/2}\| A_{K}^{-1/2} \bsigma_{s}^{\Delta} \cdot \bn_{F}\|_{F} 
	\le  \sum_{F \in \cE_{K}} h_{F}^{1/2}\| A_{K}^{-1/2}\mathbf{S}({u_{s}^{\Delta}}) \cdot \bn_{F}\|_{F}\\
	&\le  \sum_{F \in \cE_{K}} h_{F}^{-1/2}\| A_{K}^{-1/2} A_{F} \jump{u_{s}^{\Delta}}\|_{F}
	\le  \sum_{F \in \cE_{K}} h_{F}^{-1/2}\| A_{F}^{1/2} \jump{u_{s}^{\Delta}}\|_{F}.
\end{split}
\eeq
This, combining with \cref{global-efficiency}, indicates that
\beq
\begin{split}
	\|A^{-1/2} \bsigma_{s}^{\Delta}\|_{\cT_{h}} & \lesssim \tri u_{s}^{\Delta}\tri \le C \| A^{1/2} \nabla (u - u_{k}^{cg})\|_{\cT_{h}} + \osc{(f)}.
\end{split}
\eeq
This completes the proof of the Theorem.

\end{proof}


%

\section{Numerical Experiments}\label{sec:7}
In this section, we report numerical results for the
Kellogg interface problem~\cite{bruce1974poisson} and the L-shaped
benchmark problem in two dimensions, discretized by conforming
finite element methods. In three dimensions, we consider a
regularized Fichera-corner problem with a vertex singularity.
In all cases, we test polynomial degrees $k=1,2,3$, and we set
$s=k-1$ for the recovered flux.

We employ a standard adaptive refinement loop of the form
\[
\mbox{Solve} \rightarrow \mbox{Estimate} \rightarrow \mbox{Mark}
\rightarrow \mbox{Refine}  \rightarrow.
\]
At each iteration, we compute the conforming finite element solution, evaluate the
local error indicators and mark elements for refinement.
For marking, we use D\"orfler's bulk criterion with parameter
$\theta\in(0,1)$.
We sort the elements by $\eta_K^2$ in descending order and mark the
smallest set $\mathcal M$ such that
\[
\sum_{K\in\mathcal M}\eta_K^2 \;\ge\;
\theta\sum_{K\in\mathcal T_h}\eta_K^2 .
\]
The marked elements are then refined, and the process is repeated
until the stopping criterion is met.

All numerical implementations were carried out in FEniCSx \cite{baratta2023dolfinx}, which provides an efficient and reproducible finite element programming framework for assembling variational forms and solving the resulting discrete systems.

\begin{example}[Kellogg's Problem]\label{ex1}
Let $\O=(-1,1)^2$ and
 \[
 u(r,\theta)=r^{\beta}\mu(\theta)
 \]
in the polar coordinates at the origin  with
$$
	\mu(\theta)=\left\{
	\begin{array}{lll}
	\cos((\pi/2-\sigma)\beta)\cdot\cos((\theta-\pi/2+\rho)\beta) &
	\mbox{if} & 0\leq \theta \leq \pi/2,\\[1ex]
	\cos(\rho\beta)\cdot\cos((\theta-\pi+\sigma)\beta) &
	\mbox{if} & \pi/2 \leq \theta \leq \pi,\\[1ex]
	\cos(\sigma\beta)\cdot\cos((\theta-\pi-\rho)\beta) &
	\mbox{if} & \pi\leq \theta \leq 3 \pi/2,\\[1ex]
	\cos((\pi/2-\rho)\beta)\cdot\cos((\theta-3\pi/2-\sigma)\beta) &
	\mbox{if} & 3\pi/2\leq \theta \leq 2\pi,
	\end{array}
\right.
$$
where $\sigma$ and $\rho$ are numbers.
The function $u(r,\theta)$ satisfies the diffusion equation in (\ref{pde}) with $A= \a I$,
$\Gamma_N=\emptyset$, $f=0$, and
 \[
	 \a =\left\{\begin{array}{ll}
 R & \quad\mbox{in }\, (0,1)^2\cup (-1,0)^2,\\[2mm]
 1 & \quad\mbox{in }\,\O\setminus ([0,1]^2\cup [-1, 0]^2).
 \end{array}\right.
 \]
In the test problem, we choose 
$\beta=0.1$ which is corresponding to
\[
	R\approx 161.4476387975881, \quad \rho= \dfrac{\pi}{4}, \quad \mbox{and} \quad
	\sigma \approx -14.92256510455152.
\]
Note that the solution $u(r,\theta)$ is only in
$H^{1+\beta-\epsilon}(\O)$ for some $\epsilon>0$ and, hence, it is
very singular for small $\beta$ at the origin. This suggests that
refinements should be centered mostly around the origin.
\end{example}

For the Kellogg problem, we take $\theta=0.3$ and stop the adaptive
loop once the relative error is below $1\%$.
The final adaptive meshes for $k=1,2,3$ are shown in
\cref{fig:kellogg-meshes}, and the convergence of the true error and
the global estimator $\eta$ is reported in \cref{fig:kellogg-error}.
In all cases, the refinement is concentrated near the origin, where
the solution is singular.
We obtain optimal convergence for each degree.
For reference, we include the expected rate $N^{-k/2}$ and
the computed  convergence rates for the true error.
The computed convergence rates for the true error are obtained using
\texttt{polyfit} in Matlab, which performs a least-squares fit of the error data
in a log--log scale to extract the asymptotic slope. 
For all cases, we observe a consistent computed convergence rate with the optimal convergence rate; the estimator exhibits the same optimal behavior

The average efficiency index takes the values $1.3726$, $3.6363$, and
$6.5877$ for $k=1,2,3$, respectively.

\begin{figure}[ht]
  \centering
  \begin{subfigure}[t]{0.32\textwidth}
    \centering
    \includegraphics[width=\linewidth]{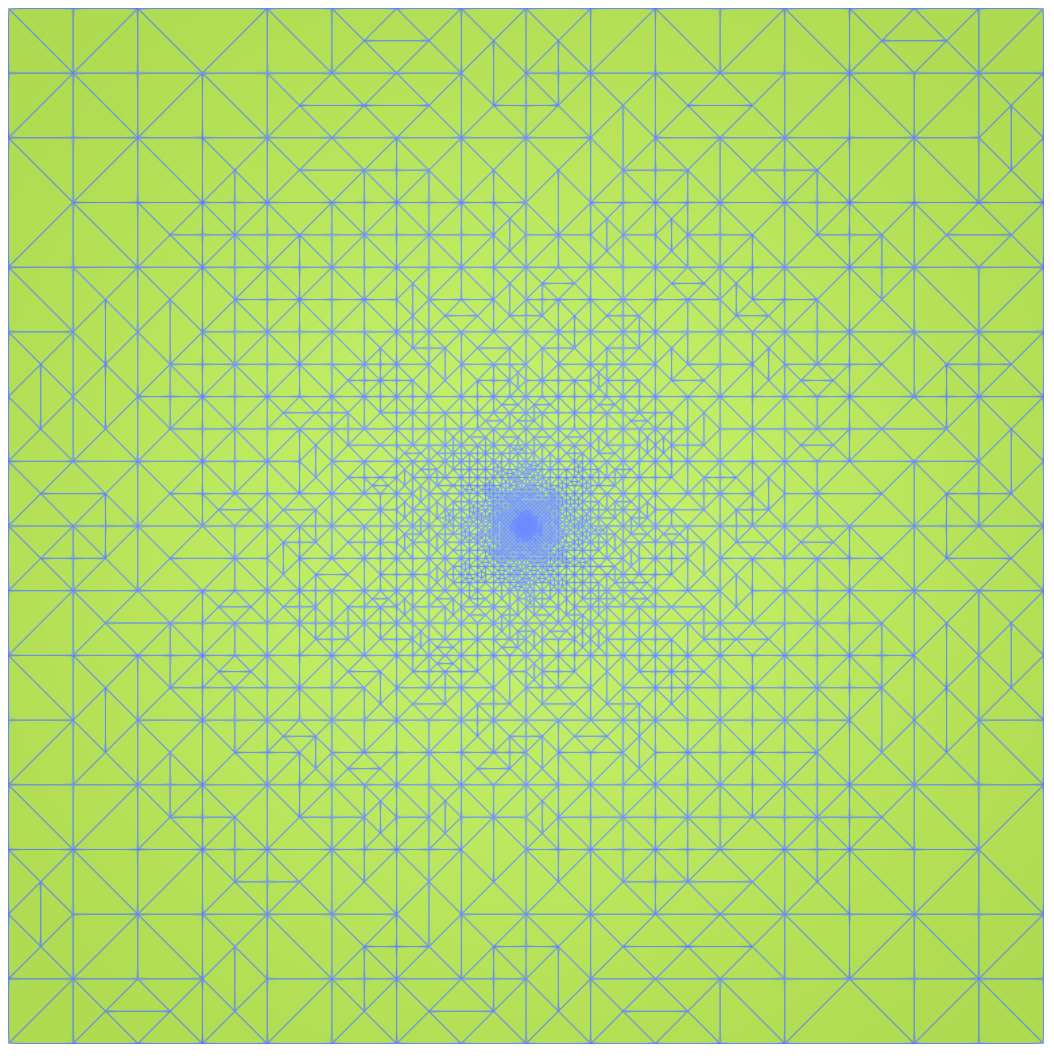}
  \end{subfigure}
  \hfill
  \begin{subfigure}[t]{0.32\textwidth}
    \centering
    \includegraphics[width=\linewidth]{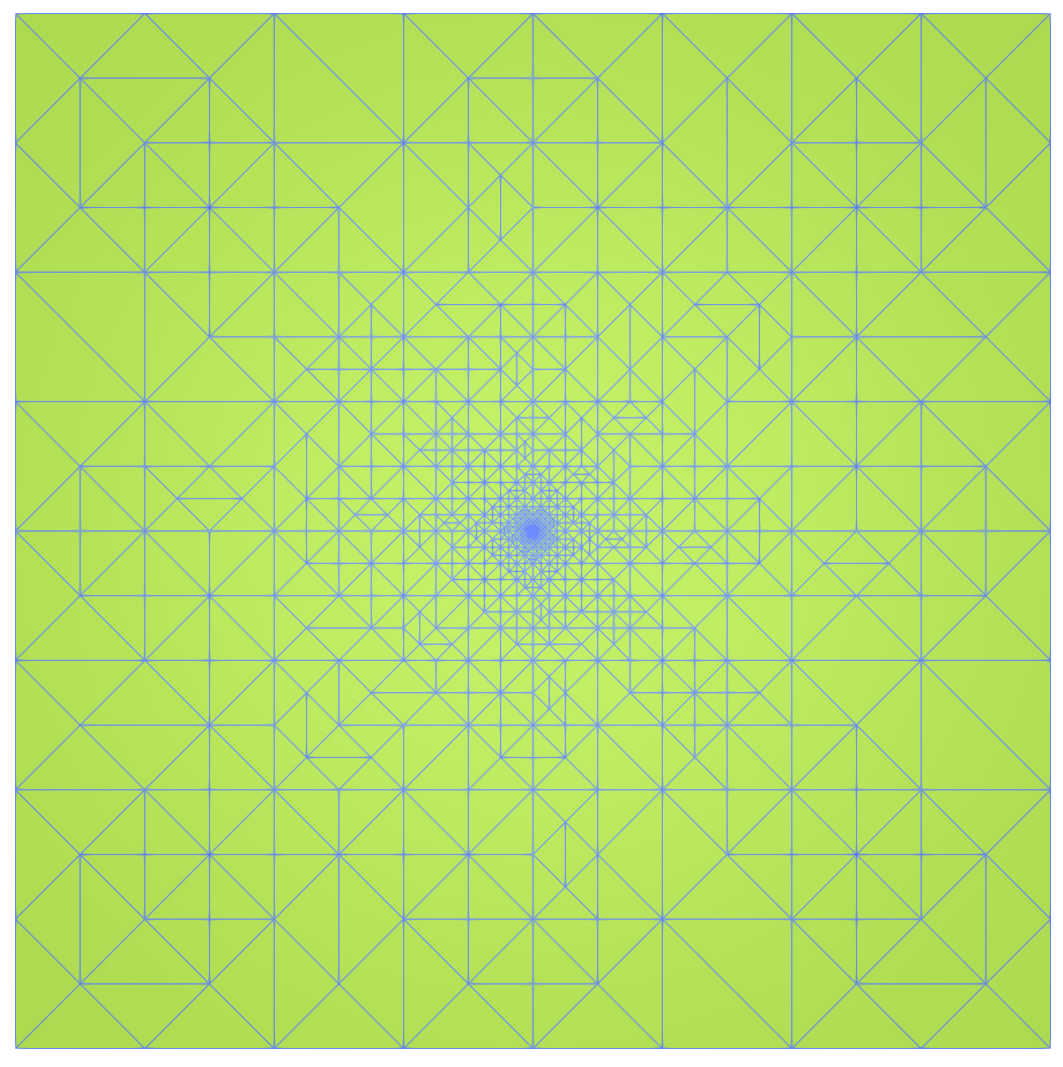}
  \end{subfigure}
    \begin{subfigure}[t]{0.32\textwidth}
    \centering
    \includegraphics[width=\linewidth]{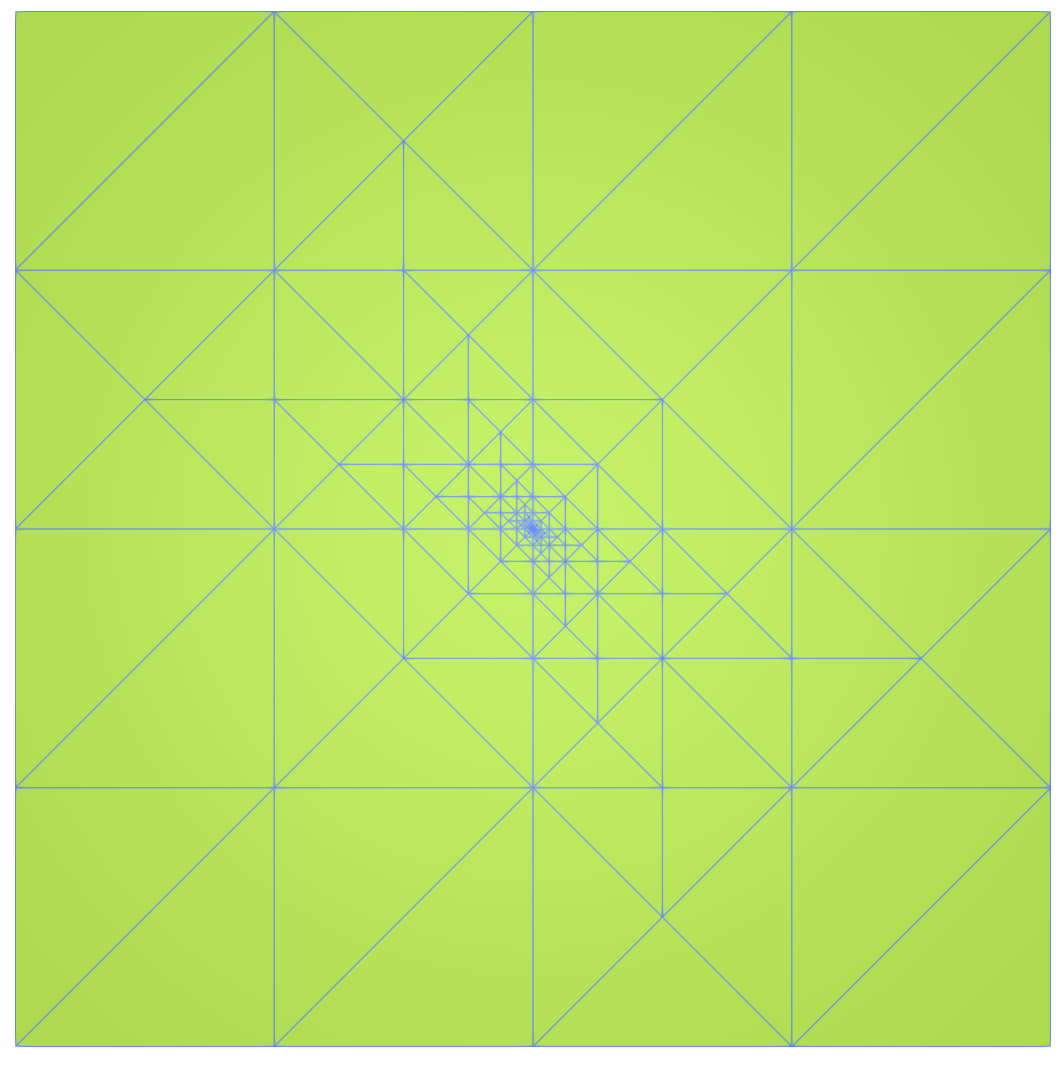}
  \end{subfigure}
  
  \caption{\cref{ex1} Final adaptive meshes for $k=1, 2, 3$}
  \label{fig:kellogg-meshes}
\end{figure}

\begin{figure}[htbp]
  \centering
  \begin{subfigure}[t]{0.32\textwidth}
    \centering
    \includegraphics[width=\linewidth]{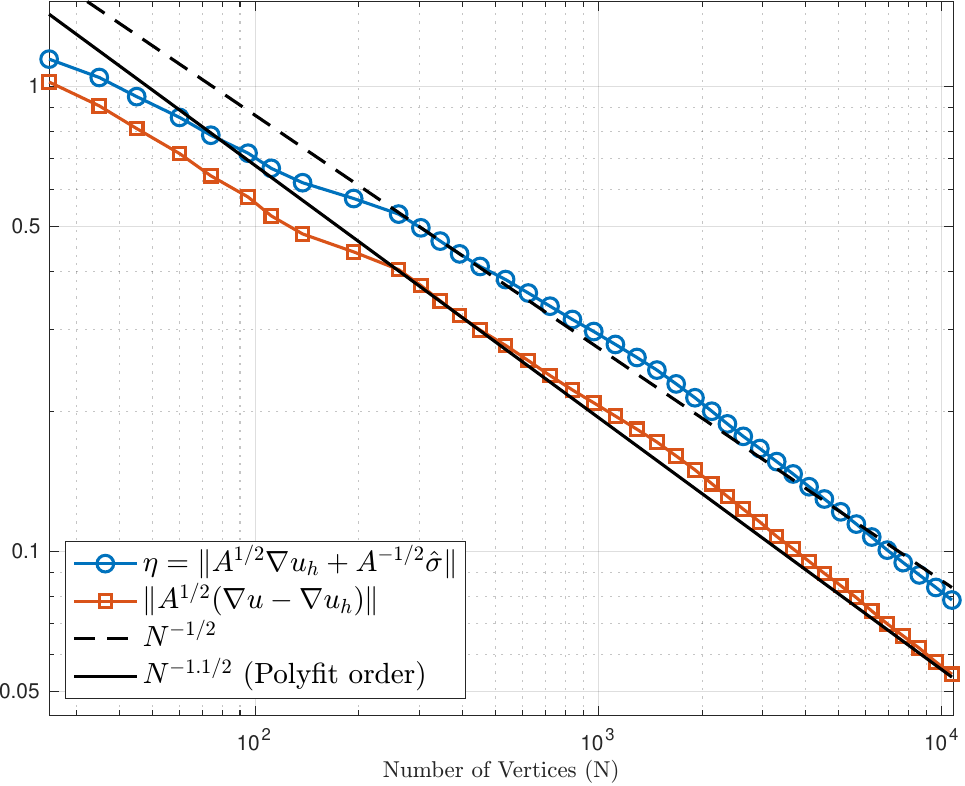}
  \end{subfigure}
  \hfill
  \begin{subfigure}[t]{0.32\textwidth}
    \centering
    \includegraphics[width=\linewidth]{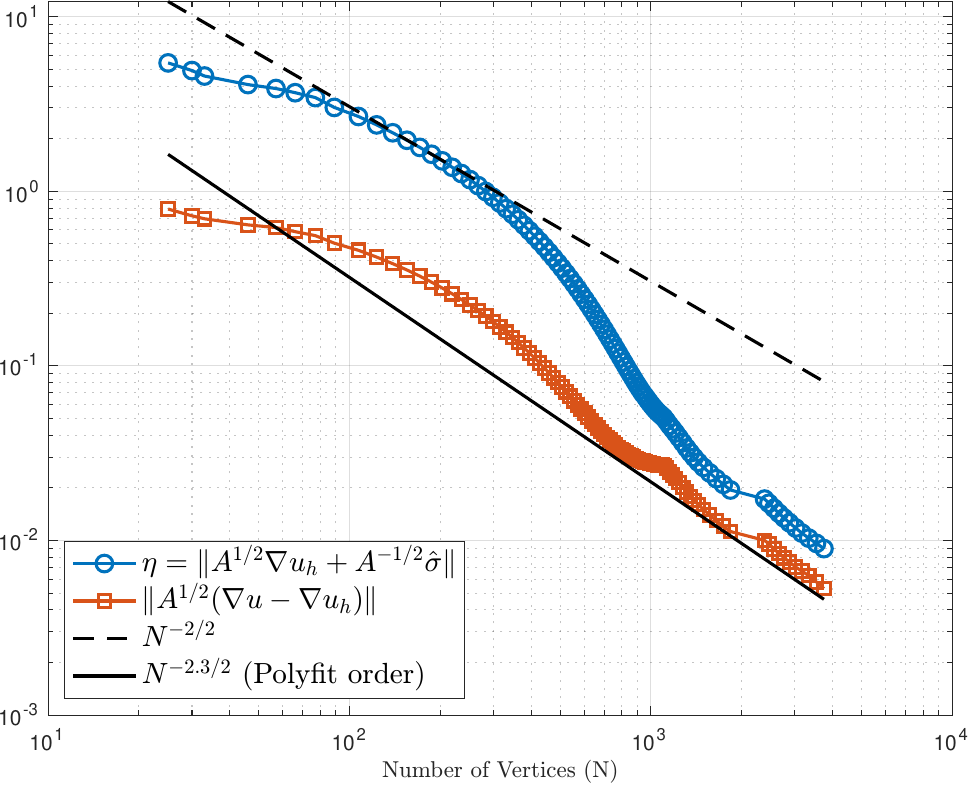}
  \end{subfigure}
    \begin{subfigure}[t]{0.32\textwidth}
    \centering
    \includegraphics[width=\linewidth]{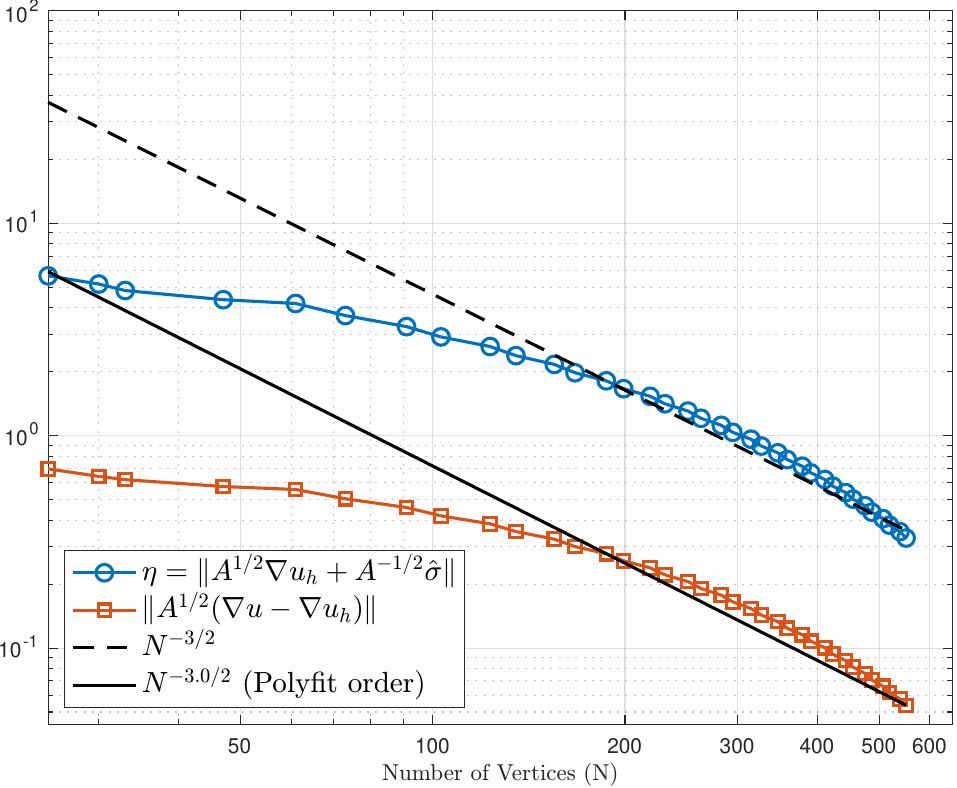}
  \end{subfigure}
  
  \caption{\cref{ex1} Adaptive error convergence results for $k=1, 2, 3$}
  \label{fig:kellogg-error}
\end{figure}


 \begin{example}[L-Shape Problem]\label{ex2}
In this example, we test the following problem: 
\[
	u(r,\theta) = r^{2/3} \sin(2 \theta /3), \quad \theta \in [0, \,3\pi/2]
\]
on the L-shaped domain $\O=(-1, \, 1)^2\setminus [0,\, 1]\times [-1,\,0]$. Note that this function satisfies (\ref{pde}) with $A(x)=I$ and $f=0$.
\end{example}
For the L-shaped problem, we set $\theta=0.2$ and use the same
stopping criterion as in \cref{ex1}, i.e., we stop once the relative
error falls below $1\%$.
The final adaptive meshes are shown in \cref{fig:2DLshape-meshes}.
For all tested degrees, refinement is concentrated near the re-entrant
corner, where the solution is singular, while the mesh remains coarse
elsewhere.
The error convergence is reported in \cref{fig:2DLshape-error}.
In this example, we observe superconvergence: for $k=2$, the observed
orders for both the true error and the estimator are close to $10$,
and for $k=3$ the observed orders are around $12$.

We emphasize that the estimator remains reliable and closely tracks
the true error throughout the adaptive process.
The efficiency index takes the values $1.12$, $1.79$, and $2.25$ for
$k=1,2,3$, respectively.


\begin{figure}[htbp]
  \centering
  \begin{subfigure}[t]{0.32\textwidth}
    \centering
    \includegraphics[width=\linewidth]{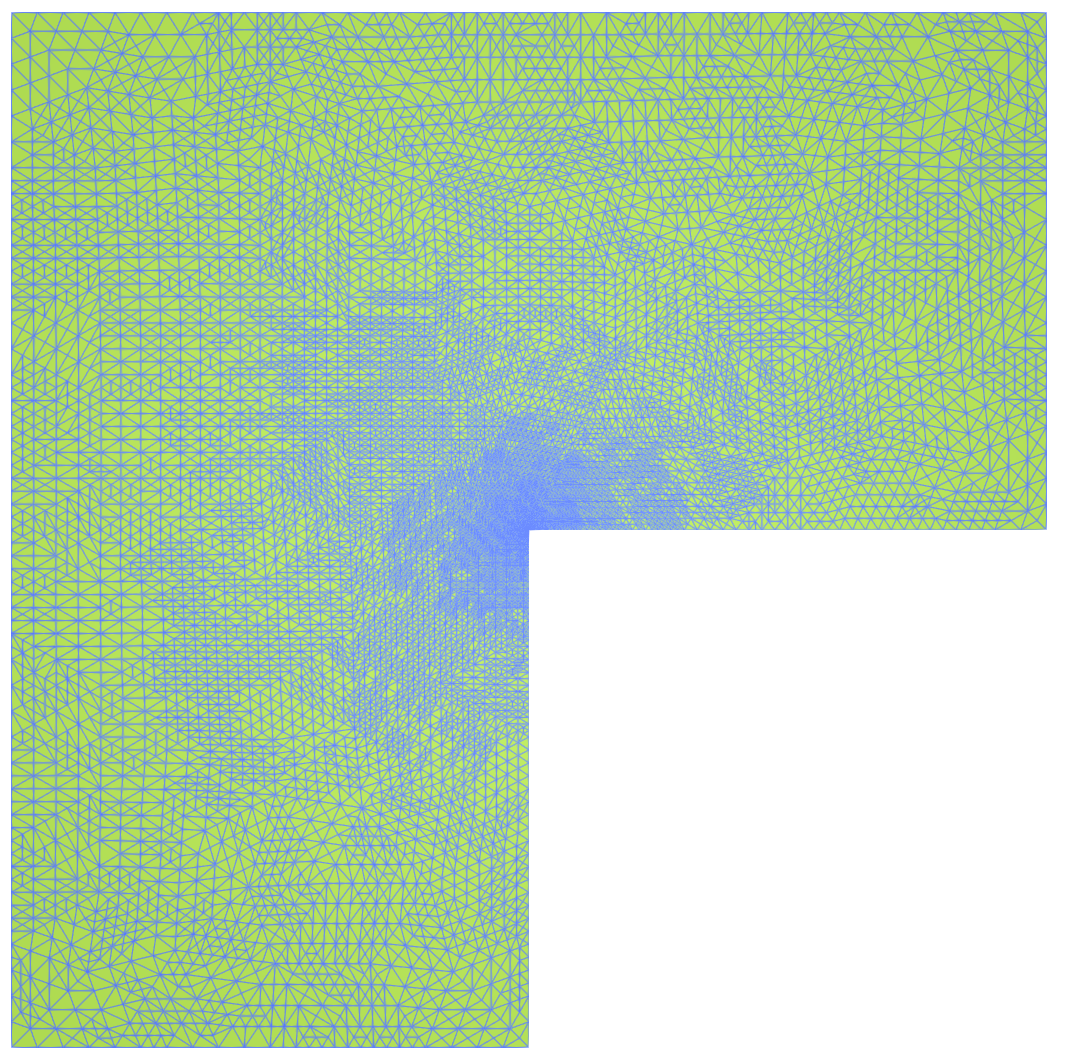}
  \end{subfigure}
  \hfill
  \begin{subfigure}[t]{0.32\textwidth}
    \centering
    \includegraphics[width=\linewidth]{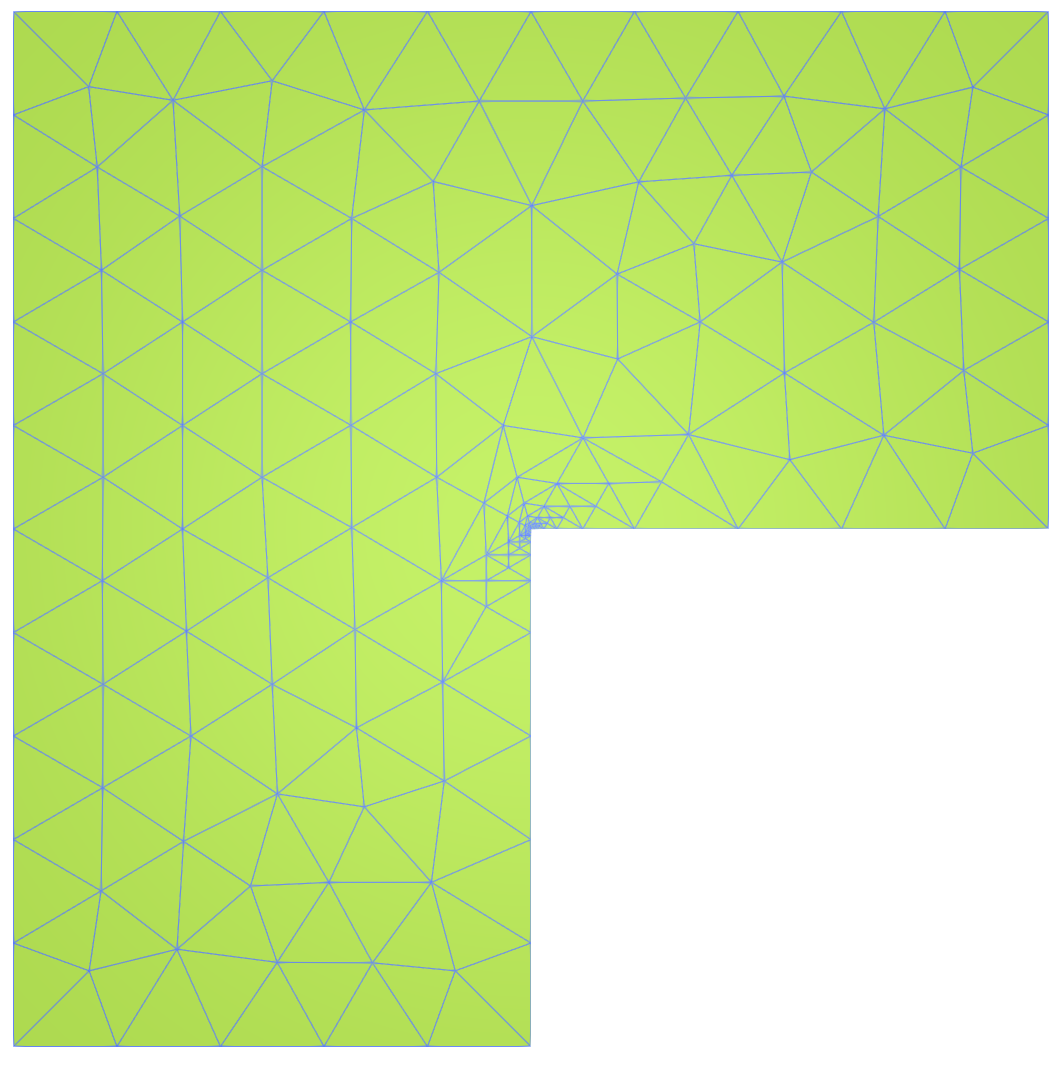}
  \end{subfigure}
    \begin{subfigure}[t]{0.32\textwidth}
    \centering
    \includegraphics[width=\linewidth]{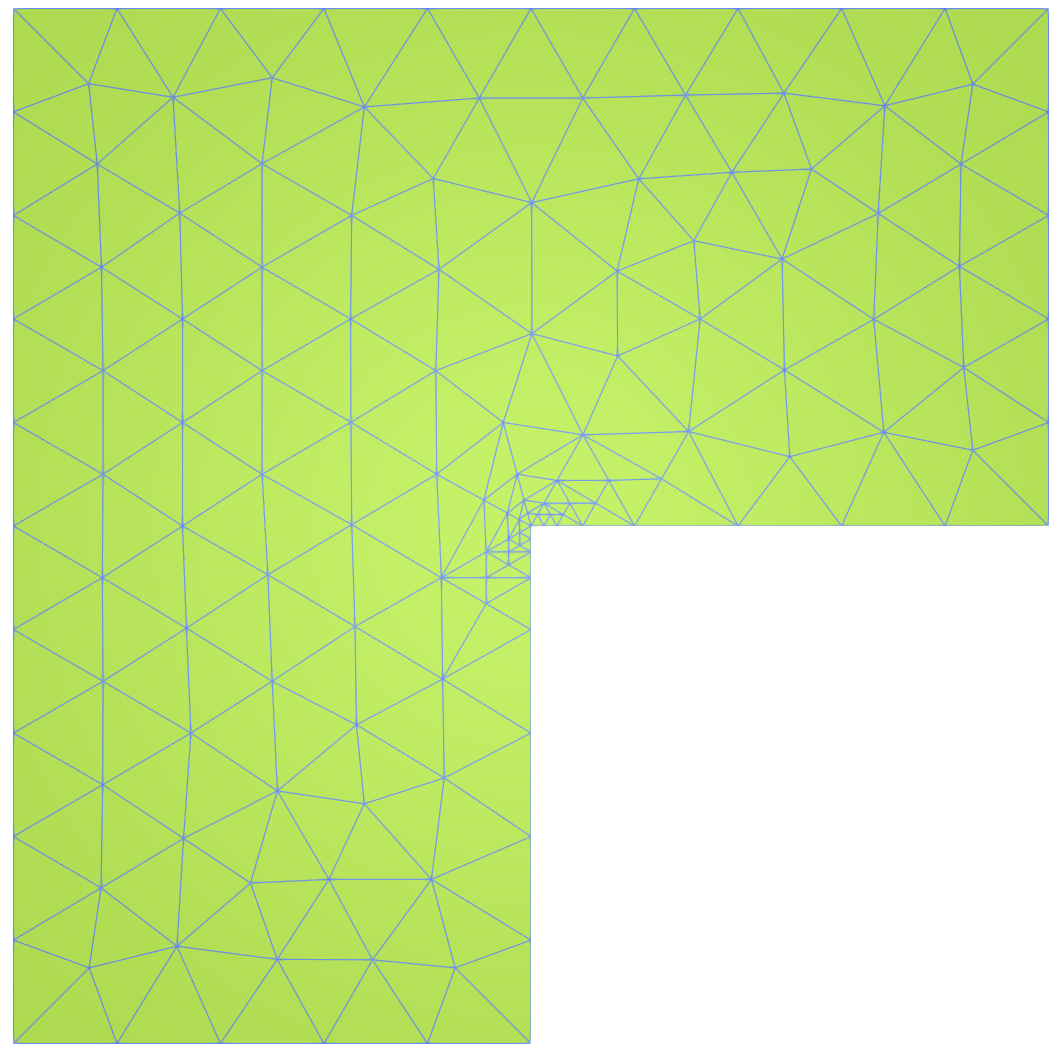}
  \end{subfigure}
  
  \caption{\cref{ex2} Final adaptive meshes for $k=1, 2, 3$}
  \label{fig:2DLshape-meshes}
\end{figure}

\begin{figure}[htbp]
  \centering
  \begin{subfigure}[t]{0.32\textwidth}
    \centering
    \includegraphics[width=\linewidth]{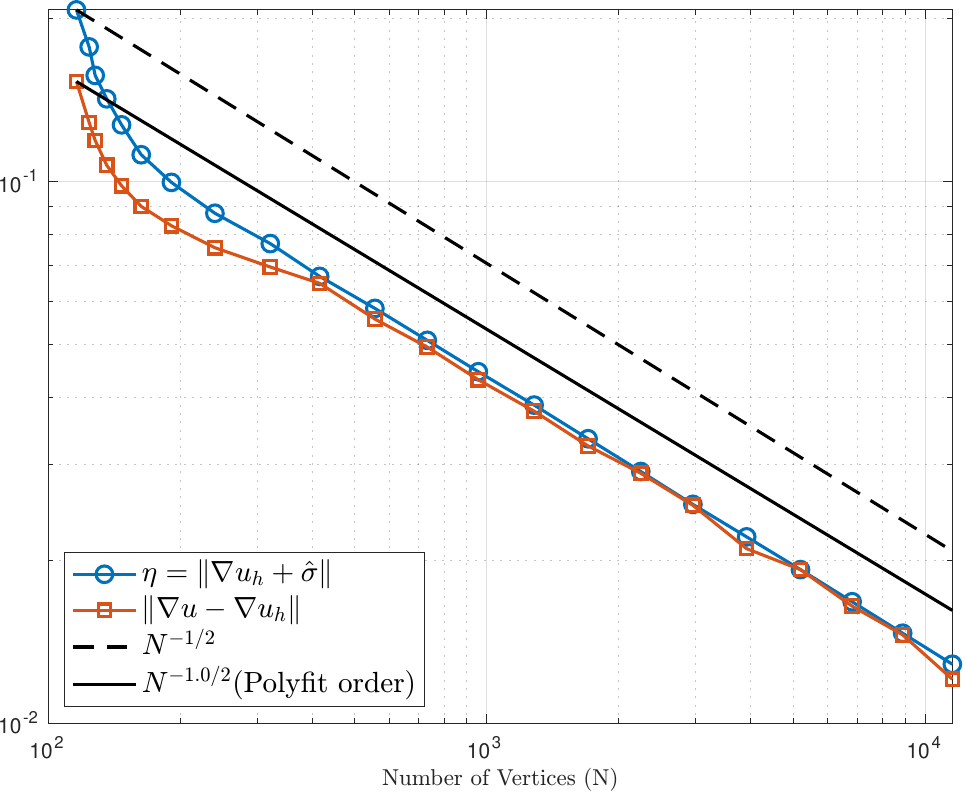}
  \end{subfigure}
  \hfill
  \begin{subfigure}[t]{0.32\textwidth}
    \centering
    \includegraphics[width=\linewidth]{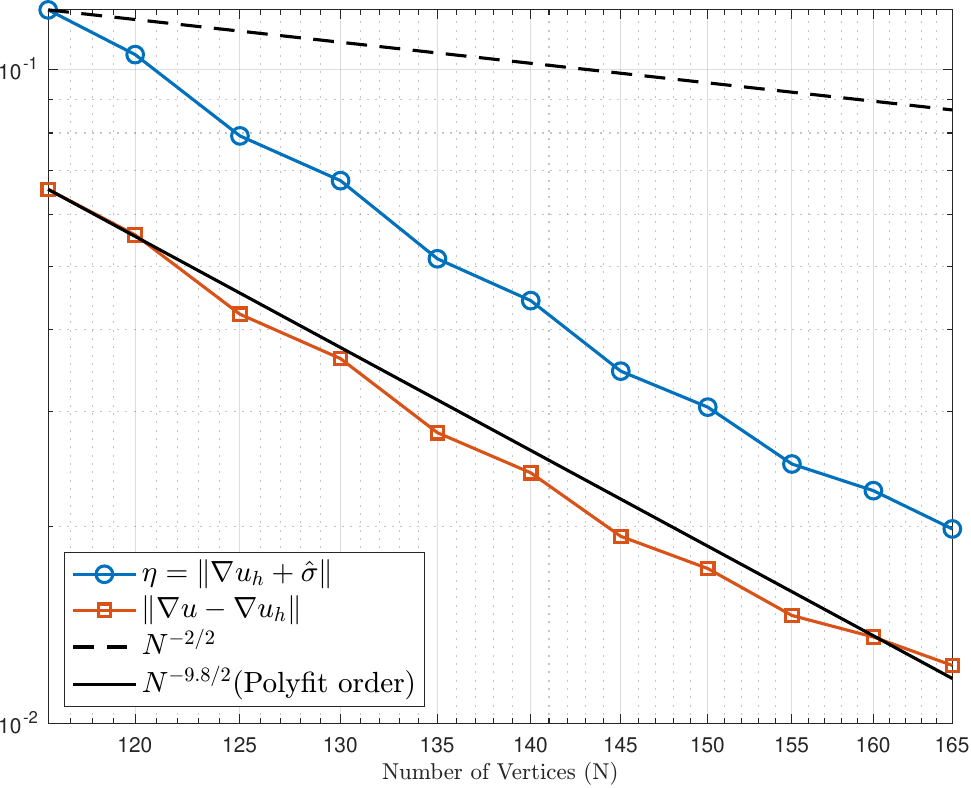}
  \end{subfigure}
    \begin{subfigure}[t]{0.32\textwidth}
    \centering
    \includegraphics[width=\linewidth]{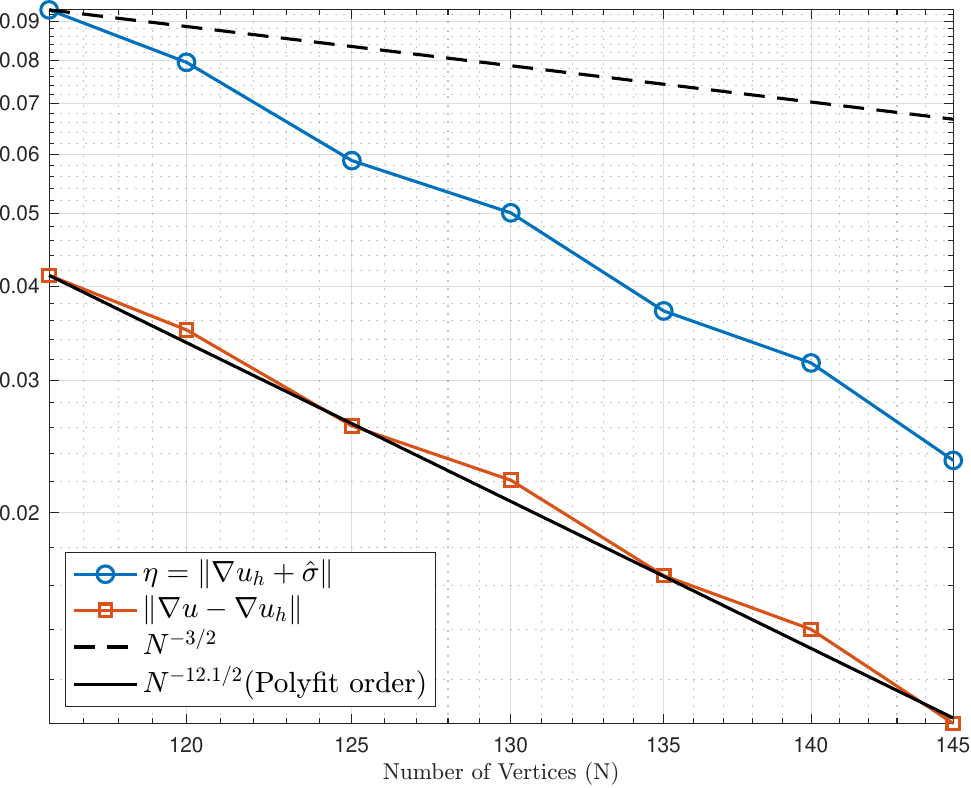}
  \end{subfigure}
  
  \caption{\cref{ex2} Adaptive error convergence results for $k=1, 2, 3$}
  \label{fig:2DLshape-error}
\end{figure}

\begin{example}[Regularized Fichera corner with vertex singularity]\label{ex3}
The Fichera corner problem \cite{mitchell2016performance} is the three-dimensional analogue of the
two-dimensional L-shaped domain.
We consider Poisson's equation on
\[
\Omega = (-1,1)^3 \setminus [0,1)^3,
\]
i.e., a cube with one octant removed.
To regularize the vertex singularity at the origin, we use the
exact solution
\begin{equation}\label{eq:fichera-f}
u(x,y,z)
=
\left(\sqrt{x^2 + y^2 + z^2 + \varepsilon}\right)^{q}, \quad q = 1/2,
\end{equation}
where $\varepsilon>0$ is a small regularization parameter and we choose $\varepsilon = 10^{-6}$. 
\end{example}

For the three-dimensional problem, we introduce a small
regularization parameter to avoid division-by-zero issues in the
right-hand side function~$f$.
Specifically, $f$ is defined consistently with the regularized exact
solution, as given in~\eqref{eq:fichera-f}.
We choose $\theta=0.15$ for $k=1,3$ and $\theta=0.3$ for $k=2$.
For $k=3$, the adaptive loop is terminated when the maximum number of
cells reaches $4500$, corresponding to a relative error
of $1.73\%$.
For $k=1$, the loop is stopped once the number of cells exceeds
$5\times10^{5}$, yielding a relative error of $3.06\%$.
For $k=2$, the refinement is terminated when the relative error drops
below $1\%$.

The final adaptive meshes are shown in \cref{fig:3DLshape-meshes}.
Consistent with the two-dimensional results, refinement is localized
near the vertex singularity.
The error convergence histories are reported in
\cref{fig:3DLshape-error}.
For $k=2$ and $k=3$, we observe slightly more than optimal convergence rates for both the
true error and the estimator.
For $k=1$, while the true error exhibits the expected optimal
convergence, the estimator converges at a slightly reduced rate,
indicating that the efficiency index is not fully uniform across all
iterations.
The slightly reduced convergence of the estimator may be
influenced by numerical effects on very fine meshes (e.g., 
solver tolerances) around the singularity, once the discretization error becomes
comparable to these errors.

\begin{figure}[htbp]
  \centering
  \begin{subfigure}[t]{0.32\textwidth}
    \centering
    \includegraphics[width=\linewidth]{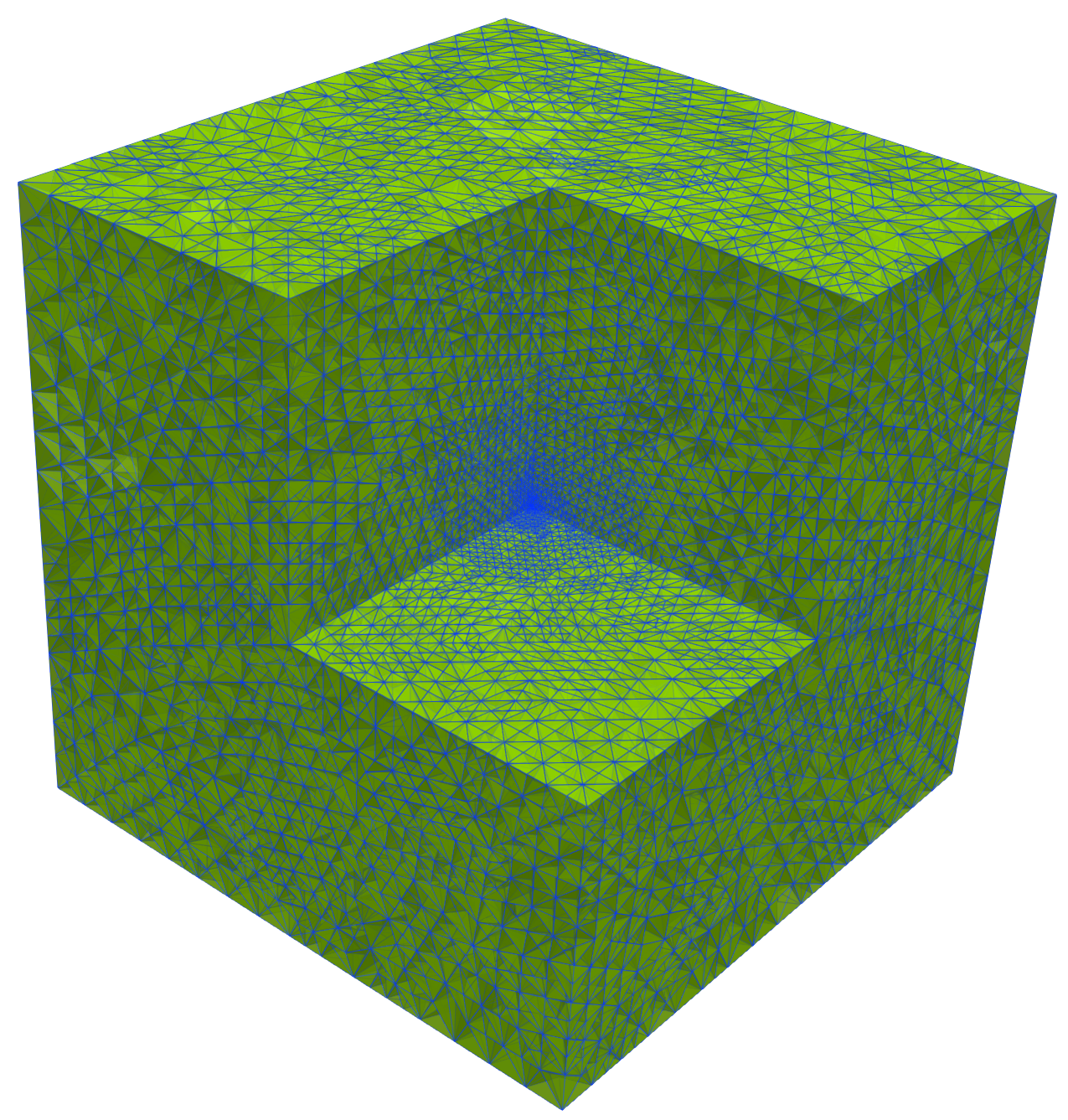}
  \end{subfigure}
  \hfill
  \begin{subfigure}[t]{0.32\textwidth}
    \centering
    \includegraphics[width=\linewidth]{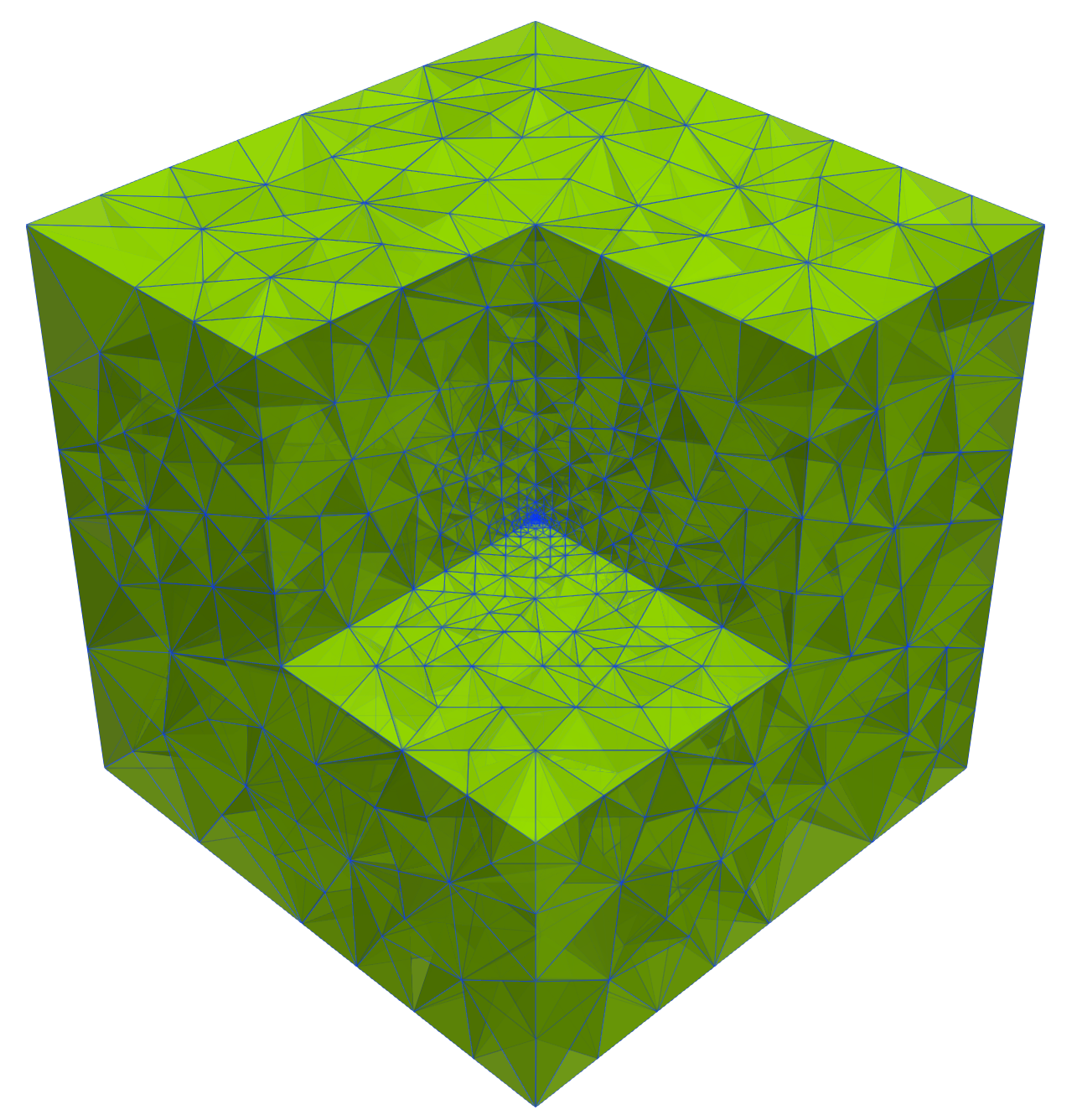}
  \end{subfigure}
    \begin{subfigure}[t]{0.32\textwidth}
    \centering
    \includegraphics[width=\linewidth]{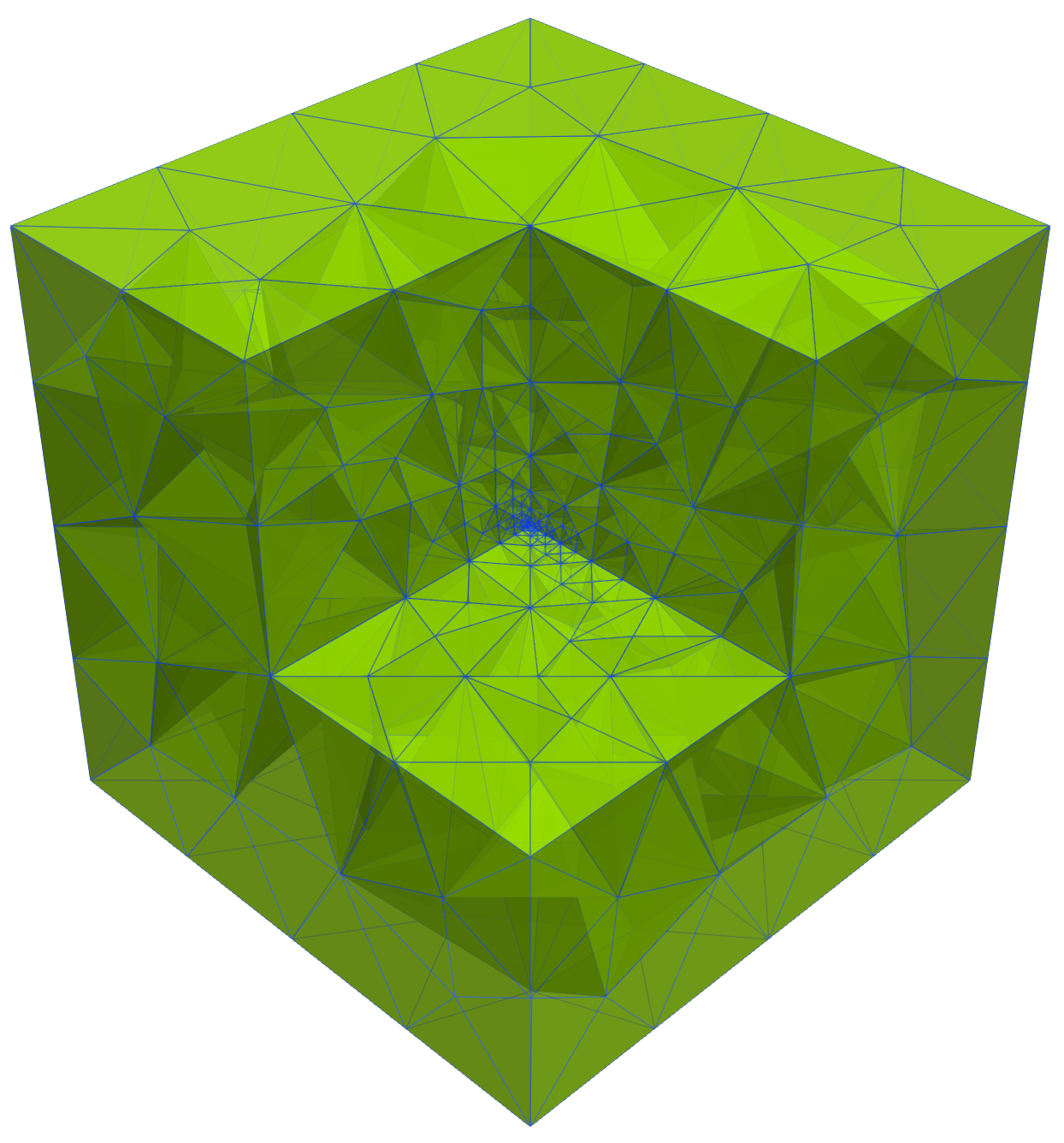}
  \end{subfigure}
  
  \caption{\cref{ex3} Final adaptive meshes for $k=1, 2, 3$}
  \label{fig:3DLshape-meshes}
\end{figure}

\begin{figure}[htbp]
  \centering
  \begin{subfigure}[t]{0.32\textwidth}
    \centering
    \includegraphics[width=\linewidth]{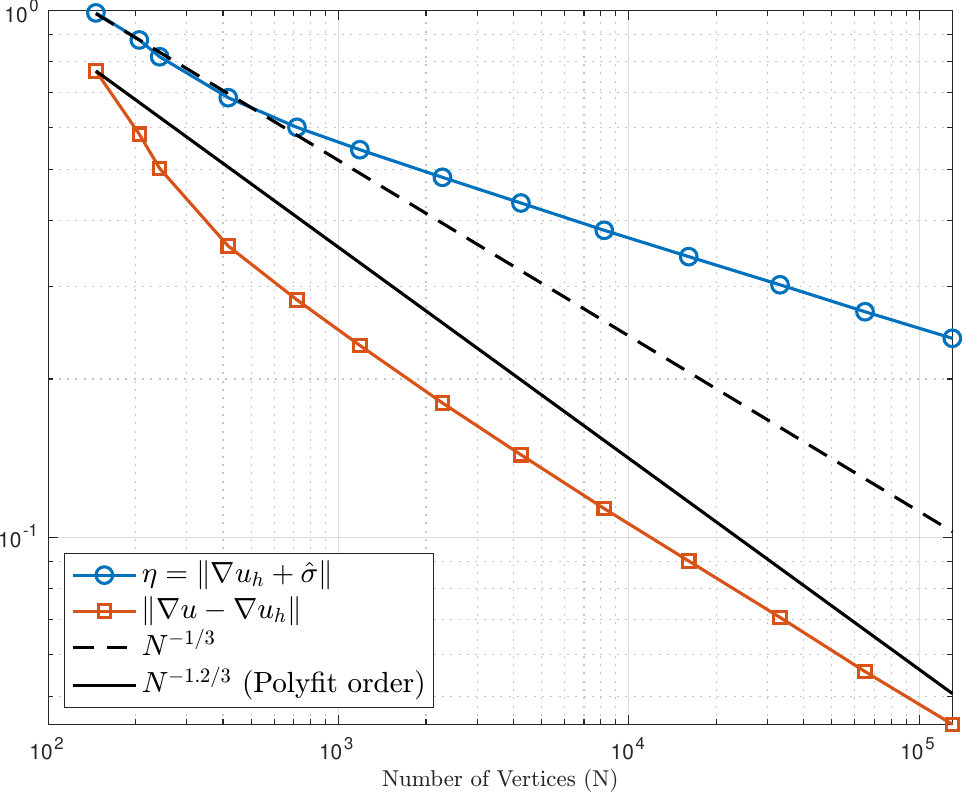}
  \end{subfigure}
  \hfill
  \begin{subfigure}[t]{0.32\textwidth}
    \centering
    \includegraphics[width=\linewidth]{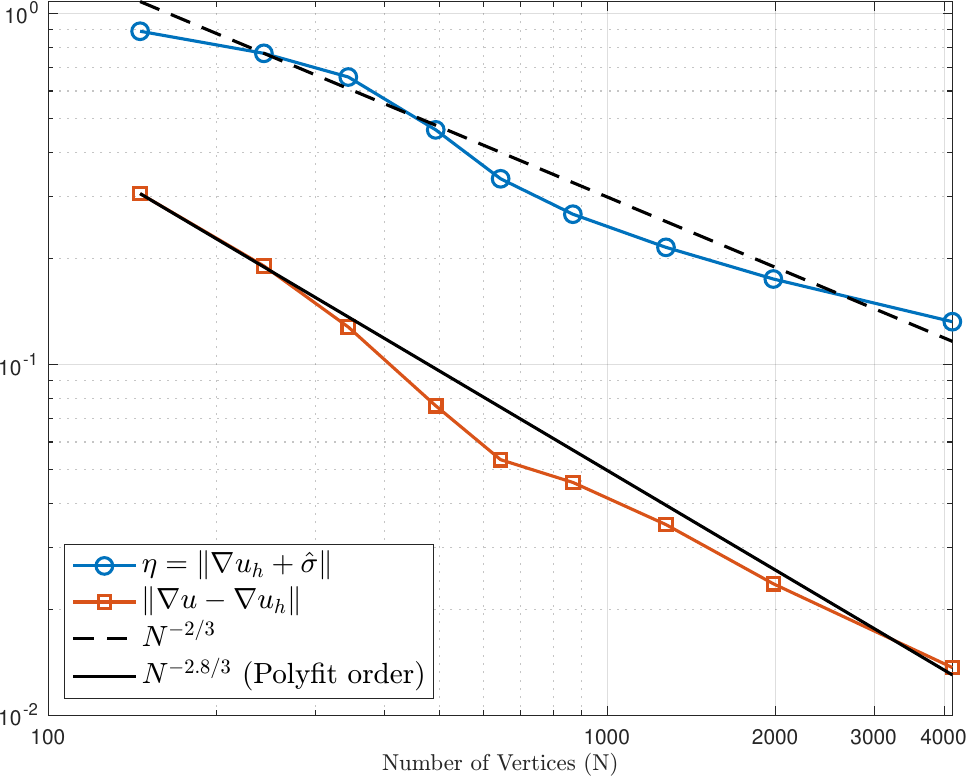}
  \end{subfigure}
    \begin{subfigure}[t]{0.32\textwidth}
    \centering
    \includegraphics[width=\linewidth]{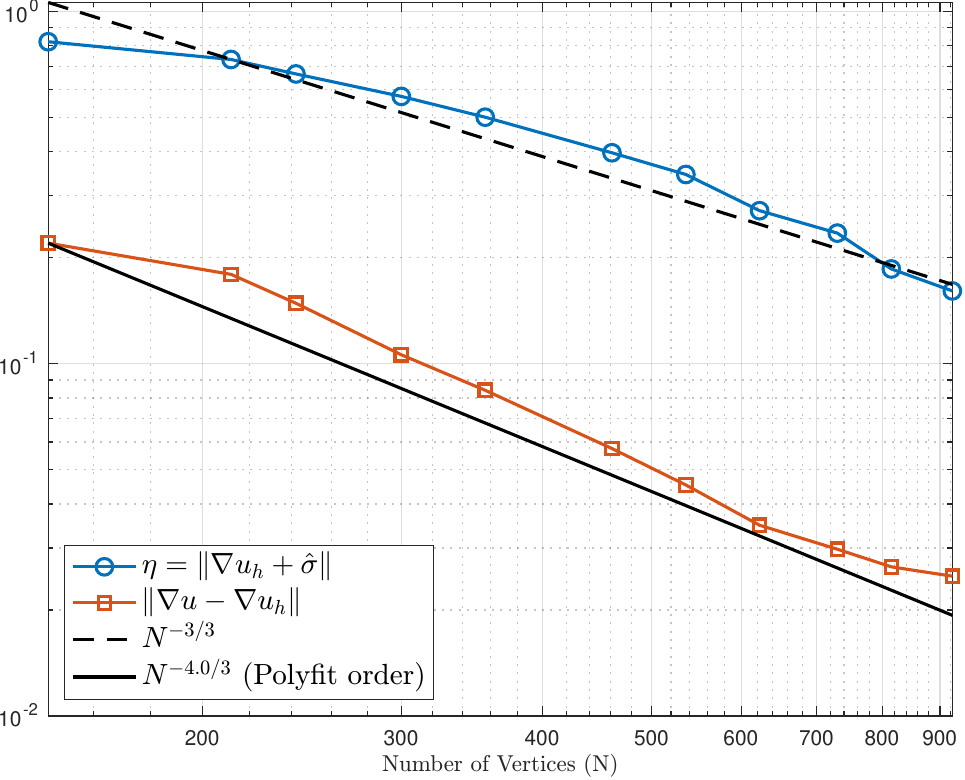}
  \end{subfigure}
  
  \caption{\cref{ex3} Adaptive error convergence results for $k=1, 2, 3$}
  \label{fig:3DLshape-error}
\end{figure}

The efficiency index takes the values $2.93$, $5.95$, and $6.09$ for
$k=1,2,3$, respectively.

\bibliographystyle{siamplain}
\bibliography{references}

\providecommand{\noopsort}[1]{}
\begin{thebibliography}{10}

\bibitem{Ai:07b}
{\sc M.~Ainsworth}, {\em A posteriori error estimation for discontinuous
  galerkin finite element approximation}, SIAM Journal on Numerical Analysis,
  45 (2007), pp.~1777--1798.

\bibitem{AiOd:93}
{\sc M.~Ainsworth and J.~T. Oden}, {\em A unified approach to a posteriori
  error estimation using element residual methods}, Numerische Mathematik, 65
  (1993), pp.~23--50.

\bibitem{ainsworth2000posteriori}
{\sc M.~Ainsworth and J.~T. Oden}, {\em A Posteriori Error Estimation in Finite
  Element Analysis}, vol.~37, John Wiley \& Sons, 2000.

\bibitem{baratta2023dolfinx}
{\sc I.~A. Baratta, J.~P. Dean, J.~S. Dokken, M.~Habera, J.~HALE, C.~N.
  Richardson, M.~E. Rognes, M.~W. Scroggs, N.~Sime, and G.~N. Wells}, {\em
  Dolfinx: the next generation fenics problem solving environment},  (2023).

\bibitem{bastian2014fully}
{\sc P.~Bastian}, {\em A fully-coupled discontinuous galerkin method for
  two-phase flow in porous media with discontinuous capillary pressure},
  Computational Geosciences, 18 (2014), pp.~779--796.

\bibitem{becker2016local}
{\sc R.~Becker, D.~Capatina, and R.~Luce}, {\em Local flux reconstructions for
  standard finite element methods on triangular meshes}, SIAM Journal on
  Numerical Analysis, 54 (2016), pp.~2684--2706.

\bibitem{bernardi2000adaptive}
{\sc C.~Bernardi and R.~Verf{\"u}rth}, {\em Adaptive finite element methods for
  elliptic equations with non-smooth coefficients}, Numerische Mathematik, 85
  (2000), pp.~579--608.

\bibitem{braess2007finite}
{\sc D.~Braess}, {\em Finite elements: Theory, fast solvers, and applications
  in solid mechanics}, Cambridge University Press, 2007.

\bibitem{BFH:14}
{\sc D.~Braess, T.~Fraunholz, and R.~H. Hoppe}, {\em An equilibrated a
  posteriori error estimator for the interior penalty discontinuous galerkin
  method}, SIAM Journal on Numerical Analysis, 52 (2014), pp.~2121--2136.

\bibitem{BrPiSc:09}
{\sc D.~Braess, V.~Pillwein, and J.~Sch{\"o}berl}, {\em Equilibrated residual
  error estimates are p-robust}, Computer Methods in Applied Mechanics and
  Engineering, 198 (2009), pp.~1189--1197.

\bibitem{braess2008equilibrated}
{\sc D.~Braess and J.~Sch{\"o}berl}, {\em Equilibrated residual error estimator
  for edge elements}, Mathematics of Computation, 77 (2008), pp.~651--672.

\bibitem{bruce1974poisson}
{\sc R.~Bruce~Kellogg}, {\em On the poisson equation with intersecting
  interfaces}, Applicable Analysis, 4 (1974), pp.~101--129.

\bibitem{CaCaZh:20}
{\sc D.~Cai, Z.~Cai, and S.~Zhang}, {\em Robust equilibrated a posteriori error
  estimator for higher order finite element approximations to diffusion
  problems}, Numerische Mathematik, 144 (2020), pp.~1--21.

\bibitem{cai2017improved}
{\sc Z.~Cai, C.~He, and S.~Zhang}, {\em Improved zz a posteriori error
  estimators for diffusion problems: Conforming linear elements}, Computer
  Methods in Applied Mechanics and Engineering, 313 (2017), pp.~433--449.

\bibitem{cai2017residual}
{\sc Z.~Cai, C.~He, and S.~Zhang}, {\em Residual-based a posteriori error
  estimate for interface problems: nonconforming linear elements}, Mathematics
  of Computation, 86 (2017), pp.~617--636.

\bibitem{cai2021generalized}
{\sc Z.~Cai, C.~He, and S.~Zhang}, {\em Generalized prager-synge identity and
  robust equilibrated error estimators for discontinuous elements}, Journal of
  Computational and Applied Mathematics,  (2021), p.~113673.

\bibitem{CaZh:11}
{\sc Z.~Cai and S.~Zhang}, {\em Robust equilibrated residual error estimator
  for diffusion problems: Conforming elements}, SIAM Journal on Numerical
  Analysis, 50 (2012), pp.~151--170.

\bibitem{capatina2024robust}
{\sc D.~Capatina, A.~Gouasmi, and C.~He}, {\em Robust flux reconstruction and a
  posteriori error analysis for an elliptic problem with discontinuous
  coefficients}, Journal of Scientific Computing, 98 (2024), p.~28.

\bibitem{capatina2016nitsche}
{\sc D.~Capatina, R.~Luce, H.~El-Otmany, and N.~Barrau}, {\em Nitsche's
  extended finite element method for a fracture model in porous media},
  Applicable Analysis, 95 (2016), pp.~2224--2242.

\bibitem{demkowicz1987adaptive}
{\sc L.~Demkowicz and M.~Swierczek}, {\em An adaptive finite element method for
  a class of variational inequalities}, in Proceedings of the Italian-Polish
  Symposium of Continuum Mechanics, Bologna, 1987.

\bibitem{destuynder1999explicit}
{\sc P.~Destuynder and B.~M{\'e}tivet}, {\em Explicit error bounds in a
  conforming finite element method}, Mathematics of Computation, 68 (1999),
  pp.~1379--1396.

\bibitem{ern2009accurate}
{\sc A.~Ern, I.~Mozolevski, and L.~Schuh}, {\em Accurate velocity
  reconstruction for discontinuous galerkin approximations of two-phase porous
  media flows}, Comptes Rendus Mathematique, 347 (2009), pp.~551--554.

\bibitem{ern2007accurate}
{\sc A.~Ern, S.~Nicaise, and M.~Vohral{\'\i}k}, {\em An accurate h (div) flux
  reconstruction for discontinuous galerkin approximations of elliptic
  problems}, Comptes Rendus Mathematique, 345 (2007), pp.~709--712.

\bibitem{ErnVo2015}
{\sc A.~Ern and M.~Vohral{\'\i}k}, {\em Polynomial-degree-robust a posteriori
  estimates in a unified setting for conforming, nonconforming, discontinuous
  galerkin, and mixed discretizations}, SIAM Journal on Numerical Analysis, 53
  (2015), pp.~1058--1081.

\bibitem{ladeveze1983error}
{\sc P.~Ladeveze and D.~Leguillon}, {\em Error estimate procedure in the finite
  element method and applications}, SIAM Journal on Numerical Analysis, 20
  (1983), pp.~485--509.

\bibitem{larson2004conservative}
{\sc M.~G. Larson and A.~J. Niklasson}, {\em A conservative flux for the
  continuous galerkin method based on discontinuous enrichment}, Calcolo, 41
  (2004), pp.~65--76.

\bibitem{mitchell2016performance}
{\sc W.~F. Mitchell}, {\em Performance of hp-adaptive strategies for 3d
  elliptic problems},  (2016).

\bibitem{oden1989toward}
{\sc J.~T. Oden, L.~Demkowicz, W.~Rachowicz, and T.~A. Westermann}, {\em Toward
  a universal hp adaptive finite element strategy, part 2. a posteriori error
  estimation}, Computer methods in applied mechanics and engineering, 77
  (1989), pp.~113--180.

\bibitem{odsaeter2017postprocessing}
{\sc L.~H. Ods{\ae}ter, M.~F. Wheeler, T.~Kvamsdal, and M.~G. Larson}, {\em
  Postprocessing of non-conservative flux for compatibility with transport in
  heterogeneous media}, Computer Methods in Applied Mechanics and Engineering,
  315 (2017), pp.~799--830.

\bibitem{petzoldt2002posteriori}
{\sc M.~Petzoldt}, {\em A posteriori error estimators for elliptic equations
  with discontinuous coefficients}, Advances in Computational Mathematics, 16
  (2002), pp.~47--75.

\bibitem{prager1947approximations}
{\sc W.~Prager and J.~L. Synge}, {\em Approximations in elasticity based on the
  concept of function space}, Quarterly of Applied Mathematics, 5 (1947),
  pp.~241--269.

\bibitem{vejchodsky2006guaranteed}
{\sc T.~Vejchodsk{\`y}}, {\em Guaranteed and locally computable a posteriori
  error estimate}, IMA Journal of Numerical Analysis, 26 (2006), pp.~525--540.

\bibitem{Ve:09}
{\sc R.~Verf{\"u}rth}, {\em A note on constant-free a posteriori error
  estimates}, SIAM journal on numerical analysis, 47 (2009), pp.~3180--3194.

\bibitem{vohralik2013posteriori}
{\sc M.~Vohral{\'\i}k and M.~F. Wheeler}, {\em A posteriori error estimates,
  stopping criteria, and adaptivity for two-phase flows}, Computational
  Geosciences, 17 (2013), pp.~789--812.

\end{thebibliography}

\end{document}